\documentclass[11pt,a4paper]{article}

\usepackage{authblk}

\usepackage[T1]{fontenc}
\newcommand{\changefont}[3]{
\fontfamily{#1} \fontseries{#2} \fontshape{#3} \selectfont}
\changefont{pcr}{b}{n}
\usepackage[utf8]{inputenc}
\usepackage{titlesec}
\usepackage{enumerate}
\usepackage{changes}
\usepackage{hyperref}
\usepackage{amsmath}
\usepackage{amsfonts}
\usepackage{amssymb}
\usepackage{amsthm}
\usepackage{newlfont}
\usepackage{mathtools}
\usepackage[top=1.5in, bottom=1.3in, left=1.2in, right=1.2in]{geometry}
\usepackage{tikz}
\usepackage[normalem]{ulem}

\newcommand{\eat}[1]{}

\theoremstyle{plain}
\newtheorem{theorem}{Theorem}[section]
\newtheorem{proposition}[theorem]{Proposition}
\newtheorem{corollary}[theorem]{Corollary}
\newtheorem{lemma}[theorem]{Lemma}

\theoremstyle{definition}
\newtheorem{definition}[theorem]{Definition}
\newtheorem*{claim}{Claim}
\newtheorem{assumption}[theorem]{Assumption}

\newtheorem*{remark}{Remark}
\newtheorem*{notation}{Notation}

\newcommand{\RR}{\mathbb{R}}
\newcommand{\NN}{\mathbb{N}}
\newcommand{\ZZ}{\mathbb{Z}}

\newcommand{\cC}{\mathcal{C}}
\newcommand{\drm}{\mathrm{d}}
\newcommand{\euler}{\mathrm{e}}

\DeclareMathOperator{\supp}{supp}

\DeclareMathOperator{\Deg}{Deg}

\DeclareMathOperator{\spec}{spec}

\DeclareMathOperator{\arsinh}{arsinh}

\DeclareMathOperator{\Eins}{\mathbf{1}}
\DeclareMathOperator{\diam}{diam}

\let\oldint\int
\renewcommand{\int}{\oldint\limits}

\def \tr {\textcolor{red}}

\newcommand{\Hmm}[1]{\leavevmode{\marginpar{\tiny%
			$\hbox to 0mm{\hspace*{-0.5mm}$\leftarrow$\hss}%
			\vcenter{\vrule depth 0.1mm height 0.1mm width \the\marginparwidth}%
			\hbox to
			0mm{\hss$\rightarrow$\hspace*{-0.5mm}}$\\\relax\raggedright #1}}}

\begin{document}
\title{Gaussian upper bounds, volume doubling and  Sobolev inequalities  on graphs}
\date{}
\author{Matthias Keller\thanks{matthias.keller@uni-potsdam.de}{ } and Christian Rose\thanks{christian.rose@uni-potsdam.de}}
\affil{Universit{\"a}t Potsdam, Institut f{\"u}r Mathematik,
14476 Potsdam} 
\maketitle
\begin{abstract}
We investigate the equivalence of Sobolev inequalities and the conjunction of Gaussian upper heat kernel bounds and volume doubling on large scales on  graphs.  For the normalizing measure, we obtain their equivalence up to constants by imposing comparability of small balls and the vertex degree at their centers. If arbitrary measures are considered, we  incorporate a new local regularity condition. Furthermore, new correction functions for the Gaussian, doubling, and Sobolev dimension are introduced. For the Gaussian and doubling, the variable correction functions always tend to one at infinity. Moreover, the variable Sobolev dimension  can be related to the doubling dimension and the vertex degree growth. 
\\
\textbf{Keywords}: graph, heat kernel, Sobolev, Gaussian bound, unbounded geometry
\\
\textbf{2020 MSC}: 39A12, 35K08, 60J74
\end{abstract}

\tableofcontents

\section{Introduction and main results}

Geometric characterizations of upper heat kernel bounds go back to the work of Varopoulos \cite{Varopoulos-85}. There it was shown that uniform diagonal upper heat kernel bounds of Dirichlet forms are characterized by uniform Sobolev inequalities. Different characterizations in terms of Nash, Gagliardo-Nirenberg, and Faber-Krahn inequalities have been obtained by several authors over the last decades in different settings, \cite{CarlenKS-87,Coulhon-92,SaloffCoste-92,Grigoryan-94,
SaloffCoste-92a,SaloffC-01,BCS, GrigoryanHH}.
In the case of continuous-time heat kernels on graphs, strong bounds on the underlying geometry have been studied since the seminal work of Davies and Pang \cite{Davies-93,Pang-93}. 
Among many other milestones, Delmotte's fundamental work on the characterization of full Gaussian bounds for heat kernels with respect to the normalizing measure is mentioned \cite{Delmotte-99}. Furthermore, the textbook \cite{Barlow-book} provides characterizations of heat kernel upper bounds on graphs with strong boundedness assumptions on the underlying geometry.
\\

In the case of graphs with unbounded geometry only partial results exist. The Davies-Gaffney-Grigor'yan estimate, an integrated heat kernel bound, has been obtained on graphs in \cite{Folz-14, BauerHuaYau-17}. Delmotte's work has been extended to graphs with normalizing measure but possibly unbounded combinatorial geometry in \cite{BarlowChen-16}. Bounded combinatorial geometry but possibly unbounded weights are studied in \cite{AndresDS-16,Bella-22} which is related to work in the continuum of Trudinger \cite{Trudinger-71}. In \cite{KellerRose-22b,KellerRose-22a} the authors obtained Gaussian upper bounds for large times assuming Sobolev and volume doubling properties on large scales in terms of intrinsic metrics. 
\\

Here we are interested in the geometric characterization of Gaussian upper bounds for the heat kernel on graphs with unbounded geometry.
We provide an equivalence of scale-invariant
Sobolev inequalities and the conjunction of  Gaussian bounds, volume doubling property on large scales, and a local regularity property.
 To allow for more unboundedness in the geometry than in the classical setting we have to expand on well known concepts using various  ideas. 

The first one is that we allow for a variable dimension in the Sobolev inequality. While this dimension function can possibly be unbounded in general, we can relate it to the volume doubling dimension and a growth rate in the case where the vertex degree is polynomially bounded.

A second idea concerns the derivation of a volume doubling property from the Sobolev inequality. Due to the discrete structure of the space additional error terms are unavoidable. However, they can be controlled for large radii.

The third idea is a local version of regularity property appearing already in a uniform form in \cite{BarlowChen-16} for the normalizing measure. Indeed, this property is an immediate consequence of the Sobolev inequality by plugging in characteristic functions. However, it allows to formulate all occurring bounds in terms of the vertex degree at centers of balls. This way this local regularity property becomes part of the characterization.

Finally, we extend the Gaussian bounds developed in   \cite{KellerRose-22a} derived from volume doubling and Sobolev inequalities. The correction terms now depend only on   degrees of the vertices for which the heat kernel is evaluated. We further emphasize that the off-diagonal bounds are indeed much finer than in previous works and are sharp in specific situations such as \cite{Davies-93,Pang-93}.\\

Although the main focus of the work are graphs of unbounded geometry, already in the special case of the normalized Laplacian, Theorem~\ref{thm:equivnormalized}, our result is new as it requires much less boundedness assumptions on the local geometry than earlier works. We discuss this result together with the set-up in the next section. Additionally we present our main result Theorem~\ref{thm:generaldeg1} for which we also extract the special case  of the counting measure in Theorem~\ref{thm:counting}. The strategy of the proof is discussed in Subsection~\ref{sec:strategy} and there the structure of the present work is discussed in detail.

\subsection{Set-up}

Let $X$ be an at most countable set and denote by $ \mathcal{C}(A) $ the real-valued functions $ f $ on $ X $ with support $ \mathrm{supp}\, f $ in $ A \subset  X $ {and by $ \mathcal{C}_{c}(X) $ we denote the functions of compact support.
Often we will use the  convention that if $f\colon W\to \RR$ for some subset $ W $ of $ X $ or $ X\times [0,\infty] $, then we let
\[
\Vert f\Vert_W:=\sup_W\vert f\vert
\]}
We extend a function
 $m\colon X\to(0,\infty)$ to a measure on $X$ via $ m(A)=\sum_{x\in A}m(x) $, $A\subset X$.
We call a symmetric $b\colon X\times X\to [0,\infty)$  such that  $b(x,x)=0$ and
\[
\deg_{x}=\sum_{y\in X}b(x,y)<\infty, \quad x\in X,
\]
a \emph{graph} over the measure space $(X,m)$. {We write $ x\sim y $ whenever $ b(x,y)>0 $ for $ x,y\in X $.} Furthermore, we denote 
\begin{align*}
	\mathrm{Deg}_{x}=\frac{1}{m(x)}\sum_{y\in X}b(x,y)=\frac{\deg_{x}}{m(x)}.
\end{align*}

A graph is called \emph{locally finite} if $\{y\in X\colon b(x,y)>0\}<\infty$ is finite for any $x\in X$. The graph is called \emph{connected} if for any $x,y\in X$ there exists a finite sequence {$x=x_0\sim x_1\sim \ldots\sim x_n=y$ which we call a \emph{path} from $ x $ to $ y $.}

It is vital to use intrinsic metrics  to deal with unbounded Laplacians on graphs, see  \cite{Davies-93a,Folz-11,GrigoryanHuangMasamune-12,BauerKH-13, HaeselerKW-13, Folz-14b, Folz-14, BauerKW-15,Keller-15, HuangKS-20, KellerLW-21}.
An \emph{intrinsic metric}  with respect to $b$ over $(X,m)$  is a non-trivial pseudo-metric $\rho\colon X\times X\to[0,\infty)$  such that 
\[
\sum_{y\in X} b(x,y)\rho^2(x,y)\leq m(x),
\]
for all  $ x\in X $.
As usual, we let $$ B_x(r)=\{y\in X\mid \rho(x,y)\leq r\},$$ $r\geq 0 $, $x\in X$. { Furthermore, we fix a vertex $ o\in X $ and denote for $ r\ge0 $
	\begin{align*}
		B(r)=B_{o}(r)
	\end{align*}
when we do not want to put focus on the center of the balls as a variable. Furthermore, we will use the following notation for the space-time cylinder
\[
Q_x(r_1,r_2)=B_x(r_2)\times [r_1,r_2],
\]
for $ x\in X $ and $ r_1,r_2\in\RR$.}

The \emph{local jump size} function  for $ x\in X $, $ r\ge 0 $, is given by 
\begin{align*}
	s_x(r):=\sup\{\rho(y,z)\colon b(y,z)>0, y\in B_x(r), z\not\in B_x(r) \mbox{ or } y=x\mbox{ and }z\neq x\} ,
\end{align*}
where $ \sup \emptyset=0 $. The local jump size includes essentially two suprema: firstly, the jump size of leaving $ x   $  and secondly the jump size of leaving the ball $ B_{x}(r) $. Either case will become relevant at a different point of our considerations.

Furthermore,  a pseudo metric $ \rho  $ is called a \emph{path metric} with {respect to the graph}  if there is $ w: X\times X\to [0,\infty] $ such that for all $ x,y\in X $
\begin{align*}
	\rho (x,y)=\inf_{x=x_{0}\sim \ldots\sim x_{n}=y}\sum_{j=1}^{n}w(x_{j-1},x_{j})
\end{align*}
{and $w(x,y)<\infty $ iff $ x\sim y $.}
Observe that the choice $ w(x,y)=(\Deg_{x}\vee\Deg_{y})^{-1/2} $ for $ x\sim y $ and $ w(x,y) =\infty$ otherwise yields an intrinsic path metric.

\begin{assumption}\label{assumption2} We assume that the intrinsic metric is a path metric whose the  distance balls are compact and we have the following {global} \emph{finite jump size} condition  $$ S:=\Vert s\Vert_{X\times [0,\infty)}<\infty .$$
\end{assumption}

Observe that as a consequence our graphs are \emph{locally finite} and \emph{connected}. Indeed, local finiteness follows from finite balls and finite jump size, while connectedness follows from the fact that $ \rho $ is a path metric taking values in $ (0,\infty) $, cf. \cite{KellerLW-21}.

An important consequence of the assumptions  above is that the metric space $ (X,\rho) $ is complete and geodesic, i.e., for any two vertices $ x,y\in X $ there is a path $x= x_{0}\sim\ldots \sim x_{n}=y $ such that $ \rho(x,y)=\rho(x,x_{j})+\rho(x_{j},y) $ for all $ j=0,\ldots,n $, see \cite[Chapter~11.2]{KellerLW-21} or \cite{KellerMuench}.

{For a locally finite graph, we consider the operator
	\[
	\Delta f(x)= \frac{1}{m(x)}\sum_{y\in X}b(x,y)(f(x)-f(y)),\qquad f\in \mathcal{C}(X),x\in X.
	\]
	With slight abuse of notation, we denote by $ \Delta\geq0 $  the Friedrichs extension of the restriction of this operator to $ \mathcal{C}_{c}(X) $} in the Hilbert space $\ell^2(X,m)=\{f\in \mathcal{C}(X)\colon \sum_{ X}m|f|^2<\infty\}$, where $ \sum_{A}g=\sum_{x\in A}g(x) $ for an absolutely summable $ g $ over $ A \subset  X $. We denote by $\Lambda:=\inf\spec(\Delta)$ the bottom of the spectrum of $\Delta$.
\\
The minimal positive fundamental solution $p\colon [0,\infty)\times X\times X\to [0,\infty)$ of the heat equation
\[
\frac{d}{dt}u=-\Delta u \quad \text{on}\ [0,\infty)\times X
\]
is called the heat kernel of the graph. It can be seen via functional calculus that $ p $ is the kernel of the semigroup $(P_t)_{t\geq0}$, where
\begin{align*}
	P_{t}=e^{-t\Delta},\qquad t\ge0.
\end{align*}
For $ x,y\in X $ and $ f\in\cC(X) $, {we let $ \nabla_{xy}f=f(x)-f(y) $ and}
\[
\vert \nabla f\vert(x):=\Bigg(\frac{1}{m(x)}\sum_{y\in X}b(x,y) (\nabla_{xy}f)^2\Bigg)^\frac{1}{2}.
\]

In the following, we introduce the conditions with which we will be dealing in this article. They generalize the well-established variants {of these notions} by introducing certain control functions. 
\begin{definition}\label{def:sobvol} Let $B\subset X$, $R_2\geq R_1\geq 0$, as well as $n,\phi,\Psi\colon X\times [R_1,R_2]\to (0,\infty)$, $\Phi\colon X\times [R_1,R_2]\times [R_1,R_2]\to (0,\infty)$, $n>2$, $\phi,\Phi\geq 1$. 
\begin{enumerate}[(S)]
\item The {\em Sobolev inequality} $S_\phi(n,R_1,R_2)$ holds in $B$, if for all $x\in B$, $r\in[R_1,R_2]$ {and $u\in\cC(B_x(r))$}, we have
\begin{equation*}
\frac{m(B_x(r))^{\frac{2}{n_x(r)}}}{\phi_x(r)r^2}
\Vert u\Vert_{\frac{2n_x(r)}{n_x(r)-2}}^2
\leq \Vert \vert\nabla u\vert\Vert_2^2+\frac{1}{r^2}\Vert u\Vert_2^2.
\end{equation*}
We abbreviate $S_\phi(n,R_1):=S_\phi(n,R_1,R_1)$. 
\item[(G)] \emph{Gaussian upper bounds} $G_\Psi(n,R_1,R_2)$ are satisfied in $ B $   if for all $t\geq R_1^2$ and all $x,y\in B$ the heat kernel has the upper bound
\begin{multline*}
p_{t}(x,y)
\leq \Psi_x(\sqrt t\wedge R_2)\Psi_y(\sqrt t\wedge R_2)
\\
\cdot
\frac{\left(1\vee S^{-2}\left(\sqrt{t^2+\rho_{xy}^2S^2}-t\right)\right)^{\frac{n_{xy}(\sqrt t\wedge R_2)}{2}}}
 {\sqrt{m(B_{x}(\sqrt {t}\wedge R_2))m(B_{y}(\sqrt {t}\wedge R_2))}}
\euler^{-\Lambda (t-t\wedge R_2^2)-\zeta\left(\rho_{xy},t\right)},
\end{multline*}
where we set $ \rho_{xy}:=\rho(x,y) $, $ n_{xy}(t):=\frac12({n_x(t)+n_y(t)}) $
and
\begin{align*}
	\zeta(r,t):=\frac{1}{S^2}\left(r S \arsinh\left(\frac{r S}{t}\right)+t-\sqrt{t^2+r^2S^2}\right),\quad r\geq 0, t>0.
\end{align*}
\item[(V)] The \emph{volume doubling} property $V_\Phi(n, R_1,R_2)$ is satisfied in $B$ if for all $x\in B$ 
\[
m(B_x(r_2))\leq \Phi_x^{r_2}(r_1)\left(\frac{r_2}{r_1}\right)^{n_x(r_2)} m(B_x(r_1)),\qquad R_1\leq r_1\leq r_2\leq R_2.
\]

\item[(L)] The \emph{local regularity property} $L_\phi(n,R_1,R_2)$ is satisfied in $ B $ if for all $x\in B$ and $r\in[R_1,R_2]$, we have 
\[
\frac{m(B_x(r))}{m(x)}
\leq \left[2\phi_x(r)\left(1+r^2\Deg_x\right)\right]^{\frac{n_x(r)}2}.
\]
\end{enumerate}
\end{definition}
\begin{notation}
 If the functions $n,\phi,\Psi,\mu$ in the definitions above are constants, we will mention their constancy explicitly.
\end{notation}

\begin{remark}By the work of Davies and Pang, the function $\zeta$ in the case $S=1$ is sharp for the normalizing measure on the integers, 
\cite{Davies-93, Pang-93}. Moreover, we have for $r>0$
\[
\zeta(r,t)\ {\simeq} \ \frac{r^2}{2t},\qquad t\to \infty,
\]
where  {$\simeq $} 
means that the left-hand side divided by the right-hand side converges to one. 
\end{remark}

\begin{remark}The local regularity property (L) can be interpreted as doubling property from balls with large radius to balls with very small radius. It should mainly be thought as $ m(B_{x}(r))/m(x)\leq C_x r^{n} $ which however fails to be equivalent for small radii $ r $.
\end{remark}

\subsection{Main results}\label{section:setup}

First, we will present our results about the normalizing measure $m=\deg$. In this special case the Laplacian $ \Delta $ is always a bounded operator on $ \ell^{2}(X,\deg) $. Furthermore, the combinatorial distance is then an intrinsic metric which we will always use in this case. In particular, the jump size is constantly equal to $ 1 $ for this metric.

\begin{theorem}[normalizing measure]\label{thm:equivnormalized}Let $ m=\deg $, {$n>2$}, $\tfrac{\diam(X)}{2}\geq R_2\geq 8 R_1\geq 512$ and
\begin{align*}
	\mu:=\sup_{x\in X}\frac{m(B_{x}(R_{1}))}{\deg(x)}<\infty.
\end{align*}
\begin{enumerate}[(i)]
\item If there is a constant $\phi>0$ such that $S_\phi (n,R_1,R_2)$ holds in $B\subset X$, then there are constants $\Phi=\Phi(\phi,n,R_1), \Psi=\Psi(\phi,n,R_1)>0$ such that $ G_{\Psi}(n,4R_1,R_2)$ and $V_{\Phi}(n,R_1,R_2)$   hold in $B$.
\item 
If there are constants $\Phi,\Psi>0$ such that $ G_{\Psi}(n,R_1,R_2)$ and $V_{\Phi}(n,R_1,R_2)$ hold in $B_o(R_2)$, then there is a constant $\phi=\phi(n,R_1,\Phi,\Psi,\mu)>0$ such that $ S_{\phi}(n,4R_1,R_2)$ holds in $o$. 
\end{enumerate}
\end{theorem}
\begin{remark}The condition $\mu<\infty$ is needed only for  statement (ii). Alternatively, we could incorporate the quantity $\sup_{x\in B(R_2)} m(B_x(R_1))/\deg(x)$ into the constant $\phi$. The appearance of such a regularity assumption can  be traced back to the local regularity property in \cite{BarlowChen-16}. 
	
Moreover, for the case of bounded combinatorial degree the constant $ \mu $ is trivially bounded. See \cite{Barlow-book} and references therein for earlier results in this direction. Indeed, previous results depend on {rather restrictive} bounds on the local geometry of the graph. Theorem~\ref{thm:equivnormalized} removes such hypotheses. 
\end{remark}
\begin{remark}
	The theorem above allows for a uniform version on $ X $ when one assumes $ B=X $ in \emph{(i)} and to assume the conditions in \emph{(ii)} for all $ o\in X $.
\end{remark}

\begin{remark}
The proof of \emph{(ii)} only utilizes the on-diagonal Gaussian bounds.
\end{remark}

\eat{Theorem~\ref{thm:equivnormalized} gives the following uniform  version.

\begin{corollary}[normalizing measure, equivalence up to constants] For all $n>2$ and all graphs 
\begin{multline*}
\exists\  R_2\geq 8R_1\geq 1024,\  \phi(n,R_1)>0\colon \ S_{\phi}(n, R_1,R_2) \mbox{ on }X
\\[1.5ex]
\Rightarrow \quad 
\exists\  \mu(n,R_1),\Psi(n,R_1)>0\colon \ V_{\mu}(n,8R_1,R_2) \quad \& \quad G_{\Psi} (n,8R_1,R_2)\mbox{ on }X,
\end{multline*}
and 
\begin{multline*}
 \exists\ R_2\geq  4R_1\geq 0, \mu(n,R_1'),\Psi(n,R_1')>0\colon \ V_{\mu}(n,R_1,R_2) \quad \& \quad G_{\Psi} (n,R_1,R_2) \mbox{ on }X
\\[1.5ex] \Rightarrow \exists\  \phi(n,R_1)>0\colon \ S_{\phi}(n,4R_1,R_2) \mbox{ on }X.
\end{multline*}
\end{corollary}}

Next, we will present our results regarding the counting measure, i.e., $m=1$, together with an intrinsic metric with jump size $S>0$.

{Allowing for more unboundedness} comes at the cost of the local regularity condition and  correction terms including the vertex degree. In particular, obtaining Sobolev inequalities from heat kernel bounds will require the adaption of the dimension in the local Sobolev inequality.
\\
Let 
 $R_{2}\geq  {4}R_1\geq0$,  $ r,R\in[R_{1},R_{2}] $, $r\leq R$, $x\in X$, and $n>2$ be {a constant}.
For the special case of counting measure, we  set the {correction terms for the volume $ \Phi $ and for the Gaussian $ \Psi $}
to be 
\[
\Phi_x^R(r)=
(1+r^2\deg_x)
^{3n^{2}\theta^n(r)},\qquad \Psi_x(r)=\Phi_x^r(r/16),
\]
(note that $\Phi^{R}_{x}$ is independent of $R$ in this case) with exponent\[
\theta^n(r)=\left(\frac{n+2}{n+4}\right)^{\kappa(r)},\qquad \kappa(r)=\left\lfloor \frac12\log_2\frac{r}{2S}\right\rfloor.
\]
Furthermore, we let the  variable dimension be given as
\[
n_o'(r)= 
n\left[1\vee\ln\left(1+r^2\|\deg\|_{B_o(r)}\right)^{\nu(r)}\right],\]
with exponent $ \nu(r)=\nu_{o}(r) $
\[
\nu_o(r)= \frac{1}{2}\frac{1}{\ln (r/r')}+{54}n\theta^{n}(r'),
\qquad
r'=
\begin{cases}
	\frac{r}{4}&:r\in [4R_1,\exp(4\vee 4R_1)),\\
	\frac{1}{4}(\ln r)^{p}&: r\in [\exp(4\vee 4 R_1),\infty),
\end{cases}
\]
with $ p={2}/{\ln(\frac{n+4}{n+2})}. $
 \begin{remark}[Behavior of the correction terms and dimension function] 
 	Classically on manifolds, the dimension $ n>2 $ in the Sobolev inequality which is derived from Gaussian upper estimates and volume doubling coincides with the volume doubling dimension. Here, the Sobolev dimension is a function $ n' $.
 	Below we discuss, that for graphs with bounded degree the dimension $ n $ can be recovered asymptotically  and in case of polynomially bounded degree the Sobolev dimension is asymptotically increased by the polynomial growth factor of the degree.
 	
 	To this end, observe that since $\theta$ is monotone decreasing, $\Phi_{x}^R(r) $ and hence  $ \Psi_x(r) $  are always bounded for large $r$ in terms $\Deg_x$ which equals $\deg_{x}$ if $ m=1 $. We discuss  the behavior of $ n' $ in the case  when $\|\Deg\|_{B_o(r)}\leq C r^k$ for $r\gg1$
 for some $ k\geq 0 $, $C\geq 1$.
 	
 To this end, we need two preliminary considerations. First,  we have and $r\gg1$
 \[
\ln ({2C}r)^{\frac{1}{2}\frac{1}{\ln (r/r')}}=\frac{\ln ({2C}r)}{2\ln \left(\frac{4r}{(\ln r)^p}\right)}
=\frac12\frac{\left(1+\frac{\ln 2C}{\ln r}\right)}{\left(1+\frac{\ln 4}{\ln r}-p\frac{\ln\ln r}{\ln r} \right)}
\longrightarrow\frac12, \qquad r\to \infty,
 \]
since $ r'=\tfrac14(\ln r)^{p} $ in this case. Secondly, since $p=2/\ln((n+4)/(n+2))$ and $1-1/\ln2<0$,
\[
0\leq \ln ({2C}r)^{\theta^n_o(r',r)}= \ln ({2C}r)^{(\frac{n+2}{n+4})^{ \lfloor\frac12\log_2 ({C'}(\ln r)^p)\rfloor}}{\leq} { C''}(\ln r)^{1-\frac{1}{\ln 2}}\longrightarrow 0,\quad r\to \infty,
\] 
{where $ C',C''>0 $ are constants depending on $ C $, $ S $ and $ n $.}

We now employ $\|\Deg\|_{B_o(r)}\leq C r^k$ for $r\gg1$. Having the definition of $ n' $ and 
$ \nu_o(r)= \frac{1}{2}\frac{1}{\ln (r/r')}+ {54}n\theta^{n}(r') $ with $ r'= {\tfrac14}(\ln r)^{p} $ for $ r\gg1 $ in mind, we obtain
\begin{multline*}
\ln\left(1+r^2\|\Deg\|_{B_o(r)}\right)^{\nu_o(r)}
\leq 
(k+2)\ln  (2Cr)^{\frac{1}{2}\frac{1}{\ln (r/r')}+ {54}n\theta^{n}_o(r',r)}
\longrightarrow 1+\tfrac{k}{2}, \quad r\to \infty.
\end{multline*}

Hence, $$ n'_{o}(r)\to n\left(1+\tfrac{k}2\right),\qquad r\to\infty .$$ Therefore, if the vertex degree is bounded, i.e., $k=0$, we recover the dimension $n$. If $k>0$, the dimension $n'$ is affected by the behaviour of $\|\Deg\|_{B_o(r)}$. 
Moreover, if $\|\Deg\|_{B_o(r)}$ grows exponentially, the dimension might become unbounded. 
 \end{remark} 
 
Given these correction terms, we have the following result for graphs with counting measure.
\begin{theorem}[Counting measure]\label{thm:counting}Let $ m $ be the counting measure, i.e., $ m=1 $,  $ n>2 $ be a constant and
${\diam(X)}/{2}\geq R_2\geq 16 R_1\geq 2048S$.  
\begin{enumerate}[(i)]
\item If there exists $\phi>0$ such that $S_\phi (n,R_1,R_2)$ holds in $B\subset X$, then there exists $A=A(n,\phi)>0$
such that $ V_{A\Phi}(n,R_1,R_2)$, $ G_{A\Psi}(n,4R_1,R_2)$, and $L_{\phi}(n,R_1,R_2)$ hold in $B$.
\item 
If there is $A>0$ such that  $ G_{A\Psi}(n,R_1,R_2)$, $V_{A\Phi} (n,R_1,R_2)$ and $L_\phi(n,R_1,R_2)$ hold in $B_o(R_2)$, then there is $\phi'=\phi'(A,\phi, n')>0$ such that $ S_{\phi'}(n',4R_1,R_2)$ holds in $o$.
\end{enumerate}
\end{theorem}

Observe that the dimension in the local Sobolev inequality is affected by the vertex degree. {Specifically,}  part (ii) of the theorem {incorporates} a variable dimension function for every vertex, which converges to the doubling dimension if  the vertex degree is bounded. This suggests that we may also  allow for a variable dimension function in Part (i) and leads us to the following version for general graphs.

Let $R_{2}\ge  {4}R_1\geq0$ and $x\in X$ and $ r,R\in[R_{1},R_{2}] $, $r\leq R$. 
For a dimension function $n\colon X\times [R_1,R_2]\to (2,\infty)$, we let the supremum over annuli {of radii} be given as
\begin{align*}
	N_x(r)=\Vert n_x\Vert_{[r/4,r]}.
\end{align*}
Furthermore, the volume doubling  and {Gaussian} correction terms are given in terms of 
\[
A_x(r)= 2^{ {43}N_x(r)^3}\phi^{8N_x(r)^2}
\]
for some constant $ \phi>0 $ and
\[
\Phi^R_x(r)=(1+r^2\Deg_x)^{3N_x(R)^2\theta^N_x(r,R)}\quad\text{and}\quad\Psi_x(r)=\Phi_x^r(r/16),
\]
with exponent
\[
 {\theta^N_x(r,R)=\left(\frac{N_x(R)+2}{N_x(R)+4}\right)^{\eta(r)},\quad \eta(r)=\eta_{x}(r)=\left\lfloor \frac12\log_2\frac{r}{2\|s_x\|_{[r/2,r]}}\right\rfloor,}
\]
where {we recall the jump size $ s_{x}(R) $}   for leaving the ball $ B_{x}(R) $.
Next, we set 
\[ 
r'=
\begin{cases}
	r/ {4} &\colon r\in[{4}R_{1},\exp({4\vee 4 R_1})),\\
	(\ln r)^{p(r)} {/4}&\colon r\geq\exp({4\vee 4 R_1}),
\end{cases} \qquad p(r)=\frac{2}{\ln\left(1+\frac{2}{\| N\|_{B_x(r)}+2}\right)}
,\]
{and denote } $Q(r)=Q_{o}(r)=Q_{o}(r',r){=B_{o}(r)\times[r',r]}$, and the dimension function
\[
n_o'(r)= 
\Vert N\Vert_{Q_o(r)}\left[1\vee\ln\left(1+r^2\|\Deg\|_{B_o(r)}\right)^{\nu(r)}\right],\]
with exponent
\[
\nu(r)=\nu_o(r)= \frac{1}{2}\frac{1}{\ln (r/r')}+{54}\| N\|_{Q(r)}\theta^{\| N\|_{Q(r)}}_o(r',r).
\]
Finally, we let the variable Sobolev constant be given as
\[
\phi_o'(r)= 2^{ {796}\| N\|_{Q(r)}^2+\frac{2\| N\|_{Q(r)}}{\| N\|_{Q(r)}-2}}
\phi^{{145}\| N\|_{Q(r)}}.
\]

{The backbone} of this article is the following new characterization of Gaussian upper heat kernel bounds on graphs.

\begin{theorem}[General locally regular case]\label{thm:generaldeg1}Let 
$\diam(X)/2\geq R_2\geq  {4} R_1\geq 0$, $ B\subset X $ { such that $1024 \| s\|_{Q_{x}(r/4,r)}\leq r$ for all $r\in[4R_1,R_2]$}, $x\in  B $, and $ \phi\ge 1 $ a constant. 
\begin{enumerate}[(i)]
\item If  $S_\phi (n,R_1,R_2)$  in $B$, then $L_\phi(N,R_1,R_2)$, $ V_{A\Phi}(N,R_1,R_2)$,   $ G_{A\Psi}(N,4R_1,R_2)$  in $B$.
\item 
If $L_\phi(N,R_1,R_2)$, $V_{A\Phi} (N,R_1,R_2)$ and  $ G_{A\Psi}(N,R_1,R_2)$  in $B=B_o(R_2)$, then one has $ S_{\phi'}(n',{4} R_1,R_2)$  in $o$.
\end{enumerate}
\end{theorem}

\begin{remark}[{Normalizing measure}]
Applying Theorem~\ref{thm:generaldeg1} to the normalizing measure together with the discussion of the behavior of the correction terms shows that the additional assumption in Theorem~\ref{thm:equivnormalized} on $ \mu $ can be traded for the local regularity property (L) if we allow for a variable dimension. The Sobolev dimension function then converges to the doubling dimension at infinity since $ \|\mathrm{Deg}\|_{B_{o}(r)} =1$ in this case.
\end{remark}

\begin{remark}[Diameter restriction]The restriction of the radius being less than the intrinsic diameter of $X$ is only of technical nature. In fact, instead one can consider all balls $B=B_o(r)\subset X$ such that $X\setminus B_o(r-s_o(r))\neq\emptyset$. 
\end{remark}

\begin{remark}[Jump size condition]
		The assumption  $1024 \| s\|_Q\leq r$ can be avoided by choosing $ R_{1} $ large enough since the global jump size $ S $ is assumed to be finite.  However, this assumption allows us to deal with smaller $ R_{1} $ as well. In fact, for (i) we only need  $ 1024\|s(r)\|_{B}\le r $ and for (ii) we only need $ 2\|s(r/4)\|_{B_{o}(r)}\le r $, cf. Theorem~\ref{thm:general}.
\end{remark}

 \begin{remark}[Behavior of the correction terms] 
 	Some
 	 of the discussion above for counting measure also applies here. Note that due to  lack on the lower bound on the measure $ m $ and an upper bound of the vertex degree the error terms include $\Deg$ which may very well be unbounded. We also allow for a variable dimension function which may then be  unbounded as well. A mitigating factor, however, may come from the local jump size which enters the exponent $ \theta $ via $ \eta $. If the local jump size becomes small for large $ r $ -- which is to be expected for small $ m $ or large $ \Deg $ -- the exponent $ \theta $ becomes  smaller and therefore may balance the growth of the error terms. This philosophy of  decaying local jump size was already successfully applied in the study of uniqueness class results and stochastic completeness of the heat equation, \cite{HuangKS-20}.
 \end{remark}

\begin{remark}[The local regularity property] Theorem~\ref{thm:generaldeg1} is a consequence of the even more general Theorem~\ref{thm:general} which does not incorporate the local regularity property. However, the local regularity property (L) is a natural consequence of the Sobolev inequality, cf.~Lemma~\ref{lemma:gammanormalized}. Thus, it is natural to include it into the characterization of the Sobolev inequality. Indeed, a uniform version  of (L) in the spirit of $ \mu<\infty $, cf.~Theorem~\ref{thm:equivnormalized}, already appeared in \cite{BarlowChen-16}.
\end{remark}

The results of this paper are applied to specific examples including the Laguerre operator, cf. \cite{Kostenko-21} in an upcoming paper \cite{KellerKNR-24}. There we use isoperimetic estimates to infer Sobolev inequalities to employ the heat kernel bounds derived here.

\subsection{Strategy of the proofs}\label{sec:strategy}

Theorems~\ref{thm:equivnormalized},  \ref{thm:counting} and \ref{thm:generaldeg1}           are special cases of  more general technical results. These are  summarized in Theorem~\ref{thm:general} which does not involve a local regularity property. Roughly speaking we want to show an ``equivalence'' of  (G) \& (V) \& (L)
and (S). Below we will discuss the overall strategy  and the necessary adaptions for the special cases.\\

Section~\ref{section:doubling} is dedicated to show how (S) implies (V) \& (L). To deduce (V) from (S) we follow the strategy of \cite{SaloffC-01}. Inserting appropriate cut-off functions into (S) gives an intermediate step of an iteration procedure. The number of steps of this iteration procedure is restricted by the jump size. The error term for the volume doubling -- above named $ \Phi $ -- which emerges from the iteration depends now on the initial ball, the mass of the center, and the radius and the jump size, cf.~Theorem~\ref{thm:adapteddoubling}. As discussed above  $\Phi  $ can be controlled for large radii. Furthermore, (L) is an immediate consequence of (S) by plugging in characteristic functions of vertices, cf.~Lemma~\ref{lemma:gammanormalized}.

In Section~\ref{section:SG} we show how to derive (G) from (S) after we already obtained (V). The considerations of this section are a variation of the strategy developed in \cite{KellerRose-22a}. This strategy is a combination of truncated Moser iteration and Davies' method. The significant difference to  \cite{KellerRose-22a} is that our new error terms to estimate $ p_{t}(x,y) $ in (G) now only depend on the vertex degree of $ x $ and $ y $  rather than means in large balls about these vertices of radius $ \sqrt{t} $, cf.~Theorem~\ref{thm:main1sturm}.

A  strategy to derive (S) from (G) \& (V) goes back to \cite{Varopoulos-85}, cf.~\cite{SaloffC-01} and infers a weak Sobolev inequality from uniform diagonal heat kernel bounds.  A challenge in the discrete setting is that this strategy only gives a weak Sobolev inequality involving the uniform norm instead of the $ 1 $-norm.
 This issue requires the incorporation of the trivial $\ell^\infty$-$\ell^1$-embedding on balls which yields an unpleasant error in terms of reciprocals of measures within the Sobolev constant. We then conclude (S) from the weak Sobolev inequality by a version of the proofs in \cite{Barlow-book, Delmotte-97}.  The unpleasant error is still good enough in very bounded situations, cf.~Theorem~\ref{thm:heattosobolevgeneral}. However, allowing for more unboundedness we need a new idea which is the choice of a variable dimension. This choice  mitigates the possibly unbounded error terms for large radii, cf.~Theorem~\ref{thm:heattosobolevgeneralimproving} which still includes a free parameter which chosen later to prove Theorem~ \ref{thm:counting} and \ref{thm:generaldeg1} . 
 
 In Section~\ref{section:remaining} we then put Theorem~\ref{thm:adapteddoubling},~\ref{thm:main1sturm} and~\ref{thm:heattosobolevgeneralimproving} together to conclude the most general result Theorem~\ref{thm:general}. The proof for the normalizing measure, Theorem~\ref{thm:equivnormalized}, then employs  a corollary of Theorem~\ref{thm:adapteddoubling}, Theorem~\ref{thm:main1sturm} and Theorem~\ref{thm:heattosobolevgeneral}. For Theorems~\ref{thm:generaldeg1}, the strategy to conclude  (G) \& (V) \& (L)
 from (S) is similar and only invokes Lemma~\ref{lemma:gammanormalized} additionally. For the ``reverse direction'' we then have to choose the free parameter  in Theorem~\ref{thm:heattosobolevgeneralimproving}. This allows us to show  (S) with a variable dimension. Here, (L) yields that all error terms only depend on the vertex degrees and not on reciprocals of measures.

\section{Sobolev implies volume doubling and local regularity}\label{section:doubling}

In this section we obtain volume doubling properties in balls where Sobolev inequalities are satisfied. We show how the variable dimension of the Sobolev inequalities {translates into} the doubling dimension. Further, the influence of the local geometry on the doubling condition of the graph is made explicit. Finally, we discuss how to derive the local regularity property from the Sobolev inequality.

Recall that we have set $S_\phi(n,r)=S_\phi(n,r,r)$ which means that we have a Sobolev inequality with Sobolev constant $ \phi $ for the radius $ r $. We first show that the Sobolev inequality remains true when increasing the dimension $ n $. This is common knowledge but included for the convenience of the reader.

\begin{lemma}\label{lem:sobdim} Let $r>0$, $n>2$, $\phi>0$, $o\in X$ and assume $S_\phi(n, r)$ in $o$. Then $S_\phi(N,r)$ holds in $o$ for all $N\geq n$. 
\end{lemma}
\begin{proof}
Let $f\in\cC(B(r))$. From H\"older's inequality with exponents $p=\frac{n}{n-2}\frac{N-2}{N}$ and $q^{-1}=\left(\frac{n}{n-2}\frac{N-2}{N}-1\right)\frac{N}{N-2}\frac{n-2}{n}$, {i.e., $ \tfrac1p+\tfrac1q=1 $,} we get
\begin{multline*}
\frac{m(B(r))^{\frac{2}{N}}}{r^2}\Vert f\Vert_{\frac{2N}{N-2}}^2 =\frac{m(B(r))^{\frac{2}{N}}}{r^2}\left(\sum_{X} m |f|^{\frac{2N}{N-2}}\cdot 1\right)^\frac{N-2}{N}
\\
\leq 
\frac{m(B(r))^{\frac{2}{N}}}{r^2}\left(\left(\sum_{X} m |f|^{p\frac{2N}{N-2}}\right)^{\frac{1}{p}}m(B(r))^{\frac{1}{q}}\right)^{\frac{N-2}{N}}
=\frac{m(B(r))^{\frac{2}{n}}}{r^2}\Vert f\Vert_{\frac{2n}{n-2}}^2.
\end{multline*}
Since the right-hand side of $S_{\phi}(N,r)$ does not depend on $ N $, this yields the claim.
\end{proof}

Next, we derive a Nash inequality from a Sobolev type inequality which is folklore {but} included for the readers convenience.

\begin{lemma}[Sobolev to Nash]\label{lem:StoN}
Let $r>0$, $n>2$, {$ C>0 $} and $B\subset X$. Assume that for all $f\in \cC(B)$ we have 
\[
\Vert f\Vert_{\frac{2n}{n-2}}^2\leq C\left(\Vert |\nabla f|\Vert_2^2+\frac{1}{r^2}\Vert f\Vert_2^2\right).
\]
Then for all such $f$ we have 
\[
\Vert f\Vert_2^{2+4/n}\leq C \left(\Vert |\nabla f|\Vert_2^2+\frac{1}{r^2}\Vert f\Vert_2^2\right)\Vert f\Vert_1^{4/n}.
\]
\end{lemma}

\begin{proof}
Recall the Lyapunov inequality, which is a consequence of the H\"older inequality: For $1\leq p_0,p_1<\infty$, $\theta\in(0,1)$ and $p\in\RR$ such that 
\[
\frac{1}{p}=\frac{1-\theta}{p_0}+\frac{\theta}{p_1},
\]
we have for all $f\in\cC_c(X)$
\[
\Vert f\Vert_p\leq \Vert f\Vert_{p_0}^{1-\theta}\Vert f\Vert_{p_1}^\theta.
\]
If we choose $p=2$, $p_0=\frac{2n}{n-2}$, and $p_1=1$, then $\theta=\frac{2}{n+2}$ and $1-\theta=\frac{n}{n+2}$. The claim follows from applying the Sobolev inequality to the right-hand side of the resulting inequality.
\end{proof}

Recall that the local jump size for $ x\in X  $ and $ r>0 $ by
\begin{align*}
	s_x(r)=\sup\{\rho(y,z)\colon b(y,z)>0, y\in B_x(r), z\not\in B_x(r)\}.
\end{align*}

For $n\colon X\times [0,\infty)\to (2,\infty)$, $x\in X$, $R\geq r\geq 0$, let
\[
 {\tilde\theta^n_x(r,R)=\left(\frac{n_x(R)+2}{n_x(R)+4}\right)^{\tilde \eta(r)},
\qquad 
\eat{\eta(r,R)=\eta_{x}(r,R):=\left\lfloor \frac12\log_2\frac{r}{2s_x(R)}\right\rfloor}}
 {\tilde\eta(r)=\tilde\eta_{x}(r):=\left\lfloor \frac12\log_2\frac{r}{2s_x(r)}\right\rfloor}.
\]
The following lemma shows that  a Sobolev type inequality with dimension $ n $ yields a lower bound on $ m(B_{x}(r))/r^n $ in an interval.

\begin{lemma}[Non-collapsing]\label{lem:asnoncoll}Let $x\in X$,  $R\in[{2s_x(0)},\diam(X)/2]$, 
$n>2$ and $C>0$ constants, assume {$s_x(r)\leq r$, $r\in[s_x(0),R/2]$,} and  that for all $f\in \cC(B_x(R))$ we have 
\[
\Vert f\Vert_{\frac{2n}{n-2}}^2\leq C\left(\Vert| \nabla f|\Vert_2^2+\frac{1}{R^2}\Vert f\Vert_2^2\right).
\]
Then we have for all {$r\in[2s_x(0),R]$}
\[
 2^{-{6}n^2}\left(\frac{1}{ C}\right)^{\frac{n}{2}}
\left[C^{\frac{n}{2}}\frac{m(x)}{r^{n}}\right]^{{\tilde\theta}_x^n(r,R)}
\leq \frac{m(B_x(r))}{r^{n}}.
\]
\end{lemma}

\begin{proof}
Lemma~\ref{lem:StoN} yields for all $f\in \cC(B_x(R))$ the Nash inequality
\[
\Vert f\Vert_2^{2+4/n}\leq C \left(\Vert |\nabla f|\Vert_2^2+\frac{1}{R^2}\Vert f\Vert_2^2\right)\Vert f\Vert_1^{4/n}.
\]
We apply this to special cut-off functions.
For  {$r\in[2s_x(0),R]$}, choose
\[
f_r(y):=\left(\frac{r}2-\rho(x,y)\right)_+,
\]
which satisfies $\supp f_r\subset B_x(r/2)\subset B_x(R)$, 
\[
\Vert f_r\Vert_1\leq \frac{r}2 m(r/2), \quad 
 \quad \frac{r}{2}m(r/2)^{1/2}\geq \Vert f_r\Vert_2\geq \frac{r}4 m(r/4)^{1/2},
\]
where we used the notation $ m(r)=m(B_{x}(r)) $.
{Since $r/2\in[s_x(0),R/2]$ and hence $s_x(r/2)\leq r/2$ by assumption,} we have  
 $\supp\vert \nabla f_r\vert\subset B_x(r/2+s_x(r/2))\subset B_x(r)$. Thus, since $\vert \nabla f_r\vert\leq 1$, 
\[
\Vert \vert \nabla f_r\vert \Vert_2
\leq m(r)^{1/2}.
\]
Therefore, the Nash inequality applied to $f_r$ yields 
\begin{align*}
	&\left(\frac{r}4 m\left({r}/4\right)^{\frac12}\right)^{2+\frac4n}
	\leq 
	C \left(m(r)+\frac{(\frac{r}{2})^2}{R^2}m\left({r}/2\right)\right)
	\left(\frac{r}2\right)^{\frac4n}\left( m\left({r}/2\right)\right)^{\frac4n}
	\leq 2 C \ m(r)^{1+\frac4n}
	r^{\frac4n},
\end{align*}
where we used $r/2\leq R$ and the monotonicity of the measure in the last line. If we 
put 
\[\alpha=1+\frac2n,\quad \beta=1+\frac4n,\quad\text{and}\quad q=\frac\alpha\beta{=\frac{n+2}{n+4}},
\]
this is equivalent to
\begin{align*}
\left(\frac{r^2}{2 C 4^{2\alpha}}\right)^{\frac{1}{\beta}}  m\left({r}/4\right)^{q}
\leq   m(r).
\end{align*}
Iterating the above inequality we obtain
\begin{align*}
\left(\frac{r^2}{2 C 4^{2\alpha}}\right)^{\frac{1}{\beta}\sum_{i=0}^{k-1}q^i}\left(\frac14\right)^{\frac{2}{\beta}\sum_{i=1}^{k-1}iq^i} 
m\left({r}/{4^k}\right)^{q^k}
\leq   m(r).
\end{align*}
This iteration procedure yields a non-trivial lower bound on $m(r)$ as along as we have { $s_{x} (4^{-k}r/2) \leq 4^{-k}r/2$ which is by assumption satisfied as long as }{$r/4^k\geq 2s_x(0)$}, i.e., if we choose
\[
k\leq
\left\lfloor\frac12\log_2\frac{r}{2 s_x(r)}\right\rfloor =\tilde\eta(r)=:\tilde \eta 
\]
since $ s_{x}(0)\leq s_{x}(r) $.
We are left with the sums in the exponents and start with the calculation of the first from the left. Set ${\tilde \theta:=q^{\tilde\eta}}$. Clearly, geometric summation gives $\sum_{i=0}^{ {\tilde \eta}-1}q^i=\frac{1-q^{ {\tilde\eta}}}{1-q}=\frac{1-{ {\tilde\theta}}}{1-q}$. Next, since $q\in(0,1)$ we have 
\[
\sum_{{i}=1}^{k-1}iq^{i}\leq \sum_{i=0}^\infty iq^i=\frac{q}{(q-1)^2}.
\]
Hence, using $m(r)=m(B_x(r))\geq m(B_x( {s_x(r)}))\geq m(x)$ for all $r\geq 0$,
\begin{align*}
 m(B_x(r))\ge\left(\frac{r^2}{2 C 4^{2\alpha}}\right)^{\frac{1}{\beta}\frac{1- {\tilde \theta}}{1-q}}  
 4^{-\frac{2}{\beta}\frac{q}{(1-q)^2}}
m(x)^{ {\tilde \theta}}
=\left(\frac{r^2}{ C }\right)^{\frac{n}{2}(1- {\tilde \theta})} 
 {\left(\frac12\right)^{(\frac52n+4)(1-{\tilde \theta})+n(n+2)}}
m(x)^{ {\tilde \theta}}
\end{align*}
where, by noting $q=\alpha/\beta$, $\alpha=1+2/n$ and $\beta=1+4/n$, we used the trivial identities $\beta(1-q)=\beta(1-\alpha/\beta)=\beta-\alpha=2/n$ and 
$q/(1-q)=\alpha/(\beta-\alpha)=n/2+1$. 
Since $ {\tilde \theta}=q^{ {\tilde \eta}}\in(0,1)$, we get for the exponent of  $1/2$  using $ n>2 $
\eat{
\[
4(n+1)(1-\tilde\theta)+\frac{n}2(n+1)
\leq \frac12n^2+\frac{9}{2}n+4{\leq  4}n^2.
\]
}
 {\[
\left(\frac52n+4\right)(1-\theta)+n(n+2)
\leq n^2+\frac{9}{2}n+4{\leq  6}n^2.
\]}
Applying this estimate leads to the claim.
\end{proof}

\eat{Introduce for $x\in X$ and $r\geq 0$ the quantity
\[
M_x(r):=
\frac{m(B_x(r))}{m(x)}.
\]
Note that $M_x(r)\geq 1$ for all $x\in X$, $r\geq 0$.
Furthermore, }
Recall the definition of the exponent  {$ \tilde \theta_{x}^{n}(r,R) ={(\frac{n+2}{n+4})}^{\tilde \eta(r)}$} for a function $ n=n_{x}(R) $ and  {$ \tilde\eta(r)=\lfloor \frac12\log_2 (r/2s_{x}(0)) \rfloor$}.

\begin{theorem}[volume doubling]\label{thm:adapteddoubling} Let $x\in X$,  {$\diam(X)/2\geq R_2\geq R_1\geq 2s_x(0)$, $s_x(r)\leq r$, $r\in[s_x(0),R_2/2]$,} a function  $n_{x}:[R_{1},R_{2}]\to (2,\infty)$ and a constant $\phi\geq 1$. If $S_{\phi}(n,R_1,R_2)$ is satisfied in $x\in X$, then  we have $ V_{A\Phi}(n,R_1,R_2)$ in $x$, i.e., with $ n(R)=n_{x}(R) $
\[
\frac{m(B_x(R))}{R^{n(R)}}
\leq {A_{x}(R)} \Phi_x^R(r)\ \frac{m(B_x(r))}{r^{n(R)}}, \quad R_1\leq r\leq R\leq R_2,
\]
where
\[
\Phi_x^R(r)=\left[r^{n(R)}{\frac{m(B_x(R))}{R^{n(R)}m(x)}}\right]^{ {\tilde \theta}^n_x(r,R)},\quad\mbox{and}\quad A_{x}(R)=2^{ {6}n(R)^2}\phi^{n(R)}.
\]
Moreover, we obtain for all $r\in[2R_1,R_2]$ the doubling property
\begin{align*}
	m(B_x(r))\leq { A'_{x}(r)} \Phi_{x}^{r}(r/2) m(B_{x}(r/2))
\end{align*}
with $ A_{x}'(r)=2^{ {7} n(r)^2}\phi^{n(r)} $
and 
\begin{align*}
	\Phi_{x}^{r}(r/2)\leq \left({\frac{m (B_x(r))}{m(x)}}\right)^{2\vartheta(r)}\quad \mbox{ with }\;\vartheta(r):=\frac{1}{2} { {\tilde \theta}^n_x(r/8,r)}.
\end{align*}
\end{theorem}
\begin{proof}Let $R\in[R_1,R_2]$ and $ {\tilde \theta}:= {\tilde \theta}^n_x(r,R)$. The non-collapsing lemma,
Lemma~\ref{lem:asnoncoll}, yields with $C=\phi\frac{R^{2}}{m(B(R))^\frac{2}{n}}$ for all  {$r\in [2s_x(0),R]$}
\[
2^{- {6}n^2}\frac{1}{\phi^\frac{n}{2}}\frac{m(B_x(R))}{R^{n}}
\left[\phi^{\frac{n}{2}}\frac{R^{n}}{m(B_x(R))}\frac{m(x)}{r^{n}}\right] ^{ {\tilde \theta}}
\leq \frac{m(B_x(r))}{r^{n} }.
\]
Division by the second and third factor yields
\[
\frac{m(B_x(R))}{R^{n}}
\leq 
2^{ {6}n^2}
\phi^{\frac{n}{2}(1- {\tilde \theta})}
\left[{r}^{n(r)} {\frac{m(B_x(R))}{R^{n}m(x)}}\right]^{ {\tilde \theta}}\frac{m(B_x(r))}{r^{n}} .
\]
Since $ {\tilde \theta}\in(0,1)$ and $ \phi\ge1 $, we  estimate $\phi^{(1- {\tilde \theta})}\leq \phi$ to obtain the first claim.
The second claim is an application of the first one using  the definition of the constants $  {\tilde\theta} =q^{ {\tilde \eta}} $ as well as 
$q=(n_{x}(R)+2)/(n_{x}(R)+4)\in(0,1)$  and the  monotone increasing property of 
 {$  \tilde\eta(r)=\lfloor \frac12\log_2 (r/2s_{x}(r)) \rfloor $ 
\[
 {\tilde \theta}_x^{n}(r/2,r)=q^{ {\tilde \eta}(r/2)}\leq q^{ {\tilde \eta}(r/8)}= {\tilde \theta}_x^{n}(r/8,r)= 2\vartheta(r).
\]
}
This finishes the proof.
\end{proof}

The following lemma essentially appears already in \cite[Lemma~6.2]{KellerRose-22a} for the normalizing measure.

\begin{lemma}[local regularity]
\label{lemma:gammanormalized} Let $\phi\geq 1$, $n>2$ be constants, $r \ge 0$, and $x\in X$.
If $S_{\phi}(n,r)$  holds in $ x $, we have 
\[
{\frac{m(B_x(r))}{m(x)}
\leq  \left[2\phi\left(1+ r^{2}\Deg(x)\right)\right]^{\frac{n}{2}}}.
\]
In particular, if $R_2\geq R_1\geq 0$, $B\subset X$ and $S_\phi(n,R_1,R_2)$ holds in $B$, then $L_{\phi}(n,R_1,R_2)$ holds in $B$. 
\end{lemma}
\begin{proof}
This follows directly by applying $S_{\phi}(n,r)$ to $ u=1_{x} $ which is supported in $ B_{x}(r) $ by assumption.
\end{proof}

Observe that for the normalizing measure, the combinatorial graph distance is an intrinsic metric. For this metric, the jump size is always $ 1 $ which explains the uniform lower bound on the smallest radius $ R_{1} $ in the next corollary.

\begin{corollary}[volume doubling of normalized measure]\label{cor:adapteddoubling} Let $m=\deg$ and $ \rho $ be the combinatoral graph distance. Let $\phi\geq 1$, $ n>2 $ be constants   $R_2\geq R_1\geq 4$. If $S_{\phi}(n, R_1,R_2)$ holds in $x\in X$, then we have $V_\Phi(n, R_1,R_2)$ in $x$, 
where 
\eat{
\[
\Phi = 2^{\tr{8}n^2}\phi^{2n}.
\]
}
\[
\Phi = 2^{ {10}n^2}\phi^{2n}.
\]
\end{corollary}
\begin{proof}
From Theorem~\ref{thm:adapteddoubling}, we obtain $ V_{A\Phi}(n,R_{1},R_{2}) $ with $ A=2^{ {6} n^2}\phi^{n} $ and, for $ \Phi $, we observe $ m=\deg $ and obtain with Lemma~\ref{lemma:gammanormalized} and $4\le r\leq R$,  $\Deg=1$
\begin{align*}
	\Phi_{x}^{R}(r)=\left[r^{n}\frac{m(B_{x}(R))}{\deg_{x}R^{n}}\right]^{ {\tilde \theta}_{x}^{n}(r,R) }\leq \left[2\phi (r^{2}+1)\right]^{\frac{n}{2} {\tilde \theta}_{x}^{n}(r,R) }\leq  \left[4\phi r^{2}\right]^{\frac{n}{2} {\tilde \theta}_{x}^{n}(r,R)}\leq ({2}\phi)^{n}h(r)^n
\end{align*}
with  $  {\tilde \theta}_{x}^{n}(r,R)=q^{\lfloor\frac{1}{2}\log_{2}(r/2)\rfloor} \leq 1$, $ q=\frac{n+2}{n+4} \in (2/3,1)  $ and $$ h(r)= r^{q^{(\frac{1}{2}\log_{2}(r)- {\frac32})}}=\exp\left(\frac{1}{q^{3/2}}r^{(\ln q)/(2\ln2) }\ln r\right).$$
Differentiating with respect to $ r $ yields that the maximum is attained at $ r= 4^{-(1/\ln q)}$ and we can estimate $ h(r)\leq 2^{3n} $ (using $ q^{-1/\log q}=e^{-1} \le 1$, $ 1/q\le 2 $ and $ \ln (1+t)\le t $). Therefore, $ h(r)^{n}\leq  2^{3n^2} $ which yields the result.
\end{proof}

\section{Sobolev implies Gaussian bounds}\label{section:SG}

In this paragraph we adjust and expand some of the techniques from \cite{KellerRose-22a}. 
In order to obtain estimates  of solutions of the heat equation, we investigate properties of solutions of the \emph{$\omega$-heat equation}
\[
\frac{d}{dt} v_t= -\Delta_\omega v_t,
\]
where $\Delta_\omega:= \euler^\omega\Delta \euler^{-\omega}$ is a sandwiched Laplacian for $\omega\in\ell^\infty(X)$. 
The following results provide an $\ell^2$-mean value inequality for non-negative solutions of the $\omega$-heat equation. The displacement of the solutions with respect to the heat equation is measured in terms of the function
\[
h(\omega)=\sup_{x\in X}\frac{1}{m(x)}\sum_{y\in X} b(x,y)\vert \nabla_{xy}\euler^\omega\nabla_{xy}\euler^{-\omega}\vert.
\]
The semigroup $P_t^\omega:=\euler^\omega P_t\euler^{-\omega}$, $ t\ge0 $, acts on $\ell^2(X,m)$. Moreover, the map $t\mapsto P_t^\omega f$ solves the $\omega$-heat equation for all $f\in\ell^2(X,m)$ and $\omega\in\ell^\infty(X)$.
\\

In this section we abbreviate, for $B\subset X$, the probability measure 
\[
m_{B}:=\frac{1}{m(B)}m
\]
on $ B $
and $ \eta(r)=\eta_{x}(r) $ by
\begin{align*}
 {\eta(r)=\left\lfloor\frac{1}{2}\log_2\frac{r}{16 \|s_x\|_{[r/2,r]}}\right\rfloor}.
\end{align*}
Observe that $ \eta_{x}(r)\le \tilde \eta(r/8) $ where $ \tilde \eta $ was defined above Lemma~\ref{lem:asnoncoll}.

Following the proof of {\cite[Theorem~2.7]{KellerRose-22a}}, we obtain the following $\ell^2$-mean value inequality. We  indicate the necessary changes in the latter article in order to obtain our result as the line of estimates is analogous with minor modifications.
\begin{proposition}[Moser in time and space, cf.~{\cite[Theorem~2.7]{KellerRose-22a}}]\label{thm:KR22-subsolution}
Let $ x\in X $, $T\in\RR$, 
$n>2$, $\alpha=1+2/n$, $\delta\in(0,1]$, $r\geq 128 {\|s_{x}\|_{[r/2,r]}}$ and constants $\phi,\Phi\geq 1$.
Assume $S_{\phi}(n,r/2,r)$ in $ x$ and the doubling property
\[
m(B_x(r))\leq \Phi\  m(B_x(r/2)).
\]
 For all non-negative $\Delta_\omega$-subsolutions $v\geq 0$ on $[T-r^2,T+r^2]\times B_{x}(r)$, we have
 \begin{multline*}
\left(\frac{1}{2  \delta (r/2)^2}\int\limits_{ T-\delta (r/2)^2}^{T+\delta (r/2)^2}\sum_{B_{x}(r/2)} m_{B_{x}(r/2)}v_t^{2\alpha^{\eta(r)}}\drm t\right)^{\alpha^{-\eta(r)}}
\\
\leq 
 \frac{C_{n,\phi,\Phi}(1+\delta r^2h(\omega))^{\frac{n}{2}+1}  }{\delta^{\frac{n}{2}+1} r^{2}}
\int\limits_{T- \delta r^2}^{T+\delta r^2}\sum_{B_{x}(r)} m_{B_{x}(r)}v_t^{2}\drm t,
\end{multline*}
where $C_{n,\phi,\Phi}:=2^{{109}n^2}(\phi\Phi)^{2n}$.
\end{proposition}
\begin{proof}
The proof follows along the same lines as the proof of \cite[Theorem~2.7]{KellerRose-22a} with a different choice of radii for the Moser iteration steps. More precisely, using the notation of the latter article, we choose
\[
\rho_k=\frac{r}{2}\left(1+2^{-k}\right), \quad k= 0,\ldots, \eta(r).
\]
Together with the assumption on $s$, the {choice} of $\eta$ allows to carry out {$\eta(r)$} iteration steps. What remains is to {track} the constant $C_{n,\phi,\Phi}$. Following  the arguments of the proof of \cite[Theorem~2.7]{KellerRose-22a} we estimate the constants using  $ k\leq \eta(r)\le 2\eta(r) $
\[
\rho_k-\rho_{k+1}-2s_x(r)\geq r2^{-(k+3)},\quad \rho_k^2-\rho_{k+1}^2\geq r^22^{-(k+4)}.
\]
Finally, we have to replace $2^dC_D$ in the notation of the latter article by $\Phi$ from the doubling property. Since $n>2\geq 1$, we obtain a constant called $ C_{d,n} $ in \cite[Theorem~2.7]{KellerRose-22a}  where in our situation $ d=n $ and $ (1\vee C_{D})=2^{-d}\Phi $
\[C_{d,n}=\Phi\left(\Phi^{\frac{n}{2}+1}\phi^{\frac{n}{2}}\right)10^{8((n+2)(d+1)+n^2+n)+1}2^{-(1+n/2)d}
\leq [\Phi \phi]^{2n}2^{{109n^2}}
\]
and we used $ 2\le n $ as well as $ 10^{3}\leq 2^{10} $.
\end{proof}

In order to obtain subsolution estimates, we use a special case of \cite[Theorem~3.4]{KellerRose-22a} with the choices $\beta=1+2/(n+2)$, $X=\{x\}$ and $\mu=\gamma>0$. 

\begin{proposition}[Moser iteration in time, cf.~{\cite[Theorem~3.4]{KellerRose-22a}}]\label{thm:constantballsdavies}
Fix $\gamma,\delta>0$, $T\geq 0$, $n>2$, $\beta=1+2/(n+2)$, $k\in\NN_0$, and let $v\geq 0$ a bounded $\Delta_\omega$-supersolution on the cylinder $[(1-\delta)T,(1+\delta)T]\times \{x\}$. 
Then we have
\begin{align*}
\sup_{[(1-\delta/2)T,(1+\delta/2)T]\times \{x\}}\!\! v^2 \leq G \left(\frac{1}{2\delta T}\int\limits_{(1-\delta)T}^{(1+\delta)T}
\gamma\ v_t^{2\beta^k}(x)\ \drm t\right)^{\frac{1}{\beta^k}}\!\!\!,
\end{align*}
where $G=G_{x,\gamma}(\delta,T,k,n)$ is given by
\begin{align*}
G= C_{1,n}
\left[\left(1+\delta T\Deg_x\right)
\gamma^{-1}\right]^{\frac{1}{\beta^{k}}}, 
\qquad\text{and}\qquad
C_{1,n}=2^{14 n^2}.
\end{align*}
\end{proposition}
\begin{proof}
The proof of \cite[Theorem~3.4]{KellerRose-22a} with the choice $\beta=1+\frac2{n+2}=\frac{n+4}{n+2}$ reveals the constant $ C_{\beta} $
\[
 C_{\beta}=2^{\left(4+\frac{1}{\ln\frac{n+4}{n+2}}+\frac{n+2}{2}+1\right)(n+2)}.
\]
Using the mean value theorem we obtain, since $n\geq 2$,
\[
\frac{1}{\ln\frac{n+4}{n+2}}=\frac{1}{\ln(n+4)-\ln(n+2)}\leq \frac{1}{\min_{t\in[n+2,n+4]}1/t}=n+4\leq 3n.
\]
Finally, note that 
\[
\left(4+\frac{1}{\ln\frac{n+4}{n+2}}+\frac{n+2}{2}+1\right)(n+2)
\leq (2n+3n+2n)2n =14n^2.\qedhere
\]
\end{proof}

We define for $r\geq0$, $n>2$, $x\in X$, and constants $\phi,\Phi\geq 1$ the error-function $\tilde\Gamma_x(r)
:=
\tilde\Gamma_x(r,n,\phi,\Phi)\geq 0$ by
\[
\tilde\Gamma_x(r)
:=
2^{{62}n^2}(\phi \Phi)^{n}
\left[\left(1+ r^2\Deg_x
\right){\frac{m(B_{x}(r))}{m(x)}}
\right]
^{\frac{1}{2}q^{\eta(2r)}}
\]
where $ q=\frac{n+2}{n+4}\in(0,1) $   and  {$ \eta(2r)=\left\lfloor\frac{1}{2}\log_2\frac{r}{8 \|s_x\|_{[r,2r]}}\right\rfloor $}. This error term will later be estimated and split up in the error term $ A\Psi $.

\begin{theorem}\label{thm:l2meanvaluesturm} Let $ x\in X $, $r\geq 
128 {\|s_x\|_{[r/2,r]}}$, $n>2$. 
Assume that there are constants $\phi,\Phi\geq 1$ such that $S_{\phi}(n,r/2,r)$ holds in $ x $ and that the doubling property
\[
m(B_x(r))\leq \Phi\ m(B_x(r/2))
\] is satisfied. Then for all $\tau\in(0,1]$, $T\geq 0$,  
and  all non-negative $\Delta_\omega$-solutions $v$ on the cylinder $[T-r^2,T+r^2]\times B_x(r)$ we have 
\begin{align*}
v^2 _T(x)
\leq
\frac{\tilde\Gamma_x(r/2)^2(1+\tau r^2h(\omega))^{\frac{n}{2}+1}}
{\tau^{\frac{n}{2}+1}r^{2}m(B_x(r))}
\int\limits_{T-\tau r^2}^{T+ \tau r^2}\sum_{B_x(r)} m\ v_t^{2}\ \drm t.
\end{align*}
\end{theorem}
\begin{proof} Set
\[
\alpha=1+\frac{2}{n},\quad \text{and}\quad
\beta=\frac{1}{q}=\frac{n+4}{n+2}=1+\frac{2}{n+2}.
\] Clearly, $\alpha>\beta\geq 1$
hence $\alpha^{\eta(R)}\geq \beta^{\eta(R)}{\ge1}$. 
Since  $L^p$-norms with respect to probability measures are non-decreasing in $p\in[1,\infty]$,
we can use this fact after we applied  Proposition~\ref{thm:constantballsdavies} with $\gamma=m_{B_x(r/2)}(x)$ and constants $k=\eta(r)$, $\delta=\tfrac{ \tau(r/2)^2}{T}$  to get
\begin{align*}	 v_T(x)^2
&\le G\left(\frac{1}{2 \tau(r/2)^2}\int\limits_{T- \tau (r/2)^2}^{T+ \tau (r/2)^2} m_{B_x(r/2)}(x)\ v_t^{2\beta^{\eta(r)}}(x)\ \drm t\right)^{\frac{1}{\beta^{\eta(r)}}}
\\
&\le G\left(\frac{1}{2\tau (r/2)^2}\int\limits_{T- \tau (r/2)^2}^{T+ \tau (r/2)^2}\sum_{B_x( r/2)} m_{B_x( r/2)}\ v_t^{2\beta^{\eta(r)}}\ \drm t\right)^{\frac{1}{\beta^{\eta(r)}}}
\\
&\leq G\left(\frac{1}{2\tau (r/2)^2}\int\limits_{T-\tau  (r/2)^2}^{T+\tau (r/2)^2}\sum_{B_x( r/2)} m_{B_x( r/2)} v_t^{2\alpha^{\eta(r)}}\drm t\right)^{\frac{1}{\alpha^{\eta(r)}}}\\
&\leq 
C_{n,\phi,\Phi}G\frac{(1+\tau r^2h(\omega))^{\frac{n}{2}+1}}
{\tau^{\frac{n}{2}+1}r^{2}}
\int\limits_{T-\tau r^2}^{T+ \tau r^2}\sum_{B_x(r)} m_{B_x(r)}\ v_t^{2}\ \drm t,
\end{align*}
where the last estimate follows by Proposition~\ref{thm:KR22-subsolution}  with $\delta=\tau$ which is applicable since  $r\ge  128\|s_x\|_{[r/2,r]} $.
We obtain the statement since by definition we have { $C_{n,\phi,\Phi}:=2^{{109}n^2}(\phi\Phi)^{2n}$ and  $ G=G_{x,m_{B(r/2)}(x)}(\tfrac{\tau r^2}{4T},T,\eta(r),n) =2^{14 n^2}	\left[\left(1+\tfrac{\tau r^2}{4}\Deg_x\right)	\tfrac{1}{m_{B(r/2)}(x)}\right]^{\frac{1}{\beta^{\eta(r)}}}$}, such that  
\[
C_{n, \phi, \Phi}{G}=  2^{{123}n^2}(\phi\Phi)^{2n}2
\left[
\left(1+\frac{\tau r^2}4 \Deg_x
\right)
\frac{m(B_x(r/2))}{m(x)}
\right]
^{q^{\eta(r)}}
\le {\tilde\Gamma_{x}(r/2)^2},
\]
where we used
$ \tau \le 1 $ and that we squared $ \tilde \Gamma $ in the last inequality.
\end{proof}

%
%
%
%
%
%

In order to obtain the desired heat kernel bounds from the subsolution estimates for the $\omega$-heat equation, we will use the following result.

\begin{proposition}[{\cite[Theorem~5.3]{KellerRose-22a}}]\label{theorem:daviesabstractgraph}Let $T>0$,
$Y\subset X$, and
 $a, b\colon Y \to [0,\infty)$, $a\leq b$, and $\chi\colon Y\times [0,\infty)\to [0,\infty)$ such that 
for all $f\in \ell^2(X), f\geq 0$, $\omega\in\ell^\infty(X)$, and $x\in Y$ we have
\[
\chi(x, h(\omega))^{2}(P_T^\omega f)^2(x)\leq  \int_{a(x)}^{b(x)}\Vert P_t^\omega f\Vert_2^2\ \drm t.
\]
Then we have 
for all $x,y\in Y$
\begin{multline*}
p_{2T}(x,y)
\leq 
\frac
{
(b(x)-a(x))^{\frac{1}{2}}(b(y)-a(y))^{\frac{1}{2}}
\exp\left(\frac{b(x)+b(y)-2T}{2}\nu(\rho_{xy},2T)\right)
}
{\chi\big(x,\nu(\rho_{xy},2T)\big)\chi\big(y,\nu(\rho_{xy},2T)\big)}
\\
\quad \cdot\exp\big(-\Lambda(a(x)+a(y))-\zeta(\rho_{xy},2T)\big),
\end{multline*}
where 
\[
\rho_{xy}:=\rho(x,y),\quad \nu(r,t):=2 S^{-2}\left(\sqrt{1+\frac{r^2S^2}{t^2}}-1\right).
\]
\end{proposition}

The next result is a variant of \cite[Theorem~6.1]{KellerRose-22a} for varying dimensions and an error term depending only on the degree and reciprocal measure of {one vertex} rather than means of these quantities in a growing ball.

\begin{theorem}\label{thm:main1sturm}
Let $x,y\in X$, 
$\diam(X)/2\geq R_2\geq 4R_1\geq  {8(s_x(0)\vee s_y(0))}$ and assume $r\geq 1024 {\|s_z\|_{[r/2,r]}}$ for all $r\in[R_1,R_2]$ and $z\in\{x,y\}$. Let $\phi\colon \{x,y\}\times [R_1,R_2]\to [1,\infty)$ and $ n\colon\{x,y\}\times [R_1,R_2]\to(2,\infty)   $ be given  and set
\[
N\colon \{x,y\}\times [4R_1,R_2]\to(2,\infty), \quad (z,\tau)\mapsto\Vert n_z\Vert_{[\tau/4,\tau]}.
\] 
If $S_\phi(n,R_1,R_2)$ holds in $z\in\{x,y\}$,
then  $G_{A\Psi}(N,4R_1,R_2)$ holds in $z\in\{x,y\}$,
where 
\[
\Psi_z(\tau)=
 \left[(1+\tau^2\Deg_z)
M_z(\tau)\right]^{\Theta_z(\tau)},\qquad A_{z}(\tau)=2^{ {41}N_z(\tau)^3}\Vert\phi_z\Vert_{[\tau/4,\tau]}
^{2N_z(\tau)^2}
\]
with $ M_{z}(\tau)=m(B_{z}(\tau)) /m(z)$ and
\[
 \Theta_z(\tau)=3N_z(\tau)\left(\frac{N_z(\tau)+2}{N_z(\tau)+4}\right)^{ \kappa_z(\tau)},
 \qquad 
  \kappa_z(\tau):=\left\lfloor\frac{1}{2}\log_2\frac{\tau}{32\Vert s_z\Vert_{[\tau/4,\tau]}} \right\rfloor.
\]
\end{theorem}
\begin{proof}
The proof is divided into two parts which we explain before we get into the details. First we use that the Sobolev inequality implies volume doubling, Theorem~\ref{thm:adapteddoubling}. This is then used in Theorem~\ref{thm:l2meanvaluesturm} to conclude a mean value inequality.  In turn we use this together with Proposition~\ref{theorem:daviesabstractgraph} to show after an appropriate choice of the involved constants an estimate of the form 
\begin{align*}
	p_{t}(x,y)&
	\leq 2^{4{\tilde N_{xy}}}
	\tilde \Gamma_x(r/2)\tilde \Gamma_y(r/2)
	\frac{
		\left(1\vee {S^{-2}}{\left(\sqrt{t^2+\rho_{xy}^2S^2}-t\right)}\right)^{\frac{\tilde N_{xy}}{2}}}
	{\sqrt{m(B_{x}(r))m(B_{y}(r))}} 
	\euler^{-\Lambda (t- r^2)-\zeta(\rho_{xy},t)},
\end{align*}
for $ r=\sqrt{t/2}\wedge R_{2} $, where $\tilde \Gamma $ are the error terms in Theorem~\ref{thm:l2meanvaluesturm}, $ \tilde N_{xy}$ is a dimension function, $S$ is the global jump size  and $ \rho_{xy}=\rho(x,y) $ is the intrinsic metric.  From there the second part is then rather technical to further estimate the involved terms to  their desired final form.

For the first part, let $z\in\{x,y\}$, $ T\ge R_{1}^{2} $ and fix $ r=\sqrt{T}\wedge R_2 \in [R_{1},R_{2}] $. 
To ease notation we denote
 $$\tilde N:=\tilde N_{z}(r):=\Vert n_z\Vert_{[r/2,r]},\quad \tilde{\phi}:=\tilde{\phi}_{z}(r):= \Vert\phi_z\Vert_{[r/2,r]},\quad \tilde s:=\tilde{s}_{z}(r):=\Vert s_z\Vert_{[r/2,r]}.$$

First from Lemma~\ref{lem:sobdim}, we get that $S_\phi (n,R_1,R_2)$  in $z$,  implies $S_{\tilde \phi}  (\tilde N, r/2,r)$  in $z$ for every $ r\in[2R_1,R_2] $.
Then,  { {$s_x(r)\leq r$, $r\in[s_x(0),R_2/2]$}}, Theorem~\ref{thm:adapteddoubling} yields the doubling property with  constant 
\begin{align*}
	C_{z}({r})=2^{ {7}\tilde N^2}\tilde \phi_{z}(r)^{\tilde N}
	M_z(r)
	^{2\vartheta(r)},
\end{align*}
where $ \vartheta(r)=\frac{1}{2}q_{z}(r)^{\eta(r)}$ with  $
q_z(r)=\frac{\tilde N+2}{\tilde N+4}$  and $ \eta(r)=\lfloor \frac12 \log_2\frac{r}{16 \tilde s} \rfloor$.
The Sobolev inequality $S_{\tilde{\phi}}(N,r/2,r)$ and volume doubling property are the assumptions of Theorem~\ref{thm:l2meanvaluesturm}, where the doubling constant is denoted by $\Phi $, i.e., we have $ \Phi=C_{z}(r) $. We apply Theorem~\ref{thm:l2meanvaluesturm} to the function $ v $ given by
$$(t,x)\mapsto P^\omega_t f(x)$$
for $f\in \ell^2(X)$, $f\geq 0$,  $\omega\in\ell^\infty(X)$, which is an $\omega$-solution on $[0,\infty)\times X$. Hence, from Theorem~\ref{thm:l2meanvaluesturm}, we obtain
for all $\delta\in(0,1]$, $T\geq \delta r^2$, $z\in\{x,y\}$
\begin{align*}
P_T^\omega f(z)^2
\leq 
\frac{\tilde\Gamma_{z}(r/2)^2(1+\delta r^2h(\omega))^{\frac{\tilde N}{2}+1}}{\delta^{\frac{\tilde N}2+1}r^2m(B_{z}(r))}
\int\limits_{T-\delta r^2}^{T+\delta r^2}\Vert\Eins_{B_{z}(r)} P_t^\omega f\Vert_2^{2}\drm t,
\end{align*}
where $ \tilde\Gamma_x(r/2)=
2^{{62}\tilde N^2}(\tilde\phi C_{z}(r))^{\tilde N}
\left[\left(1+ (r/2)^2\Deg_z
\right)M_z(r/2)
\right]
^{\vartheta(r)} $.
\eat{\begin{multline*}
\tilde\Gamma_z(R/2)
=
2^{118N(R)^2}(\Vert\phi_z\Vert_{[R/2,R]} C_\Phi(z,R))^{2\tilde N}
\left[\left(1+ (R/2)^2\Deg_z
\right)M_z(R/2)
\right]
^{\vartheta(R)}
\\
\leq 2^{20\tilde N^3+118\tilde N^2}\Vert\phi_z\Vert_{[R/2,R]}^{2\tilde N^2+2\tilde N}
\left[1+ (R/2)^2\Deg_z
\right]
^{\vartheta(R)}
M_z(R/2)
^{4\tilde N\vartheta(R)+\vartheta(R)}
\\
\leq 2^{138\tilde N^3}\Vert\phi_z\Vert_{[R/2,R]}^{4\tilde N^2}
\left[1+ (R/2)^2\Deg_z
\right]
^{\vartheta(R)}
M_z(R/2)
^{5\tilde N\vartheta(R)}
.
\end{multline*}}

To finally apply 
Proposition~\ref{theorem:daviesabstractgraph} we choose the parameters. 
 We have $T-\delta r^2\geq 0$. We set 
\[
a(z)=T-\delta r^2,\quad b(z)=T+\delta r^2,\quad
 r(z)=r,
\]
and  $\chi(z, h(\omega))$ via
\[
\chi(z,h(\omega))^{-2}=\frac{\tilde\Gamma_{z} (r/2)^2}{\delta^{\frac{\tilde N}2+1}r^{2}
 m(B_{z}(r))}(1+\delta r^2h(\omega))^{\frac{\tilde N}{2}+1}.
\]
Proposition~\ref{theorem:daviesabstractgraph} yields
for $T\geq 4R_1^2$, and $ t=2T $
 the estimate
\begin{align*}
p_{t}(x,y)&=p_{2T}(x,y)
\\&
\leq \frac
{
(b(x)-a(x))^{\frac{1}{2}}(b(y)-a(y))^{\frac{1}{2}}
\exp\left(\frac{b(x)+b(y)-2T}{2}\nu(\rho_{xy},2T)\right)
}
{\chi\left(x,\nu(\rho_{xy},2T)\right)\chi\left(y,\nu(\rho_{xy},2T)\right)}
\\
&\quad \cdot\exp\left(-\Lambda(a(x)+a(y))-\zeta(\rho_{xy},2T)\right)
\\
&
=
\frac{2
 \tilde\Gamma_{x} (r/2)\tilde\Gamma_{y}(r/2)\delta^{-\frac{\tilde N_{xy}}{2}}}
 {\sqrt{m(B_{x}(r))m(B_{y}(r))}} \left(1+\delta r^2\nu(\rho_{xy},t)\right)^{\frac{\tilde N_{xy}}{2}+1}
 \\
 &\quad\cdot
 \exp\left(\delta r^2\nu(\rho_{xy},t)\right)
\exp\left(-\Lambda (t-2\delta r^2)-\zeta(\rho_{xy},t)\right),
\end{align*}
where $\tilde N_{xy}=(\tilde N_x(r)+\tilde N_y(r))/2$, $\nu(r,t)=2 S^{-2}\left(\sqrt{1+{r^2S^2}{t^{-2}}}-1\right) $, and $\rho_{xy}=\rho(x,y)$. Choosing 
\[
\delta=\frac{1}{2}\wedge \frac{1}{t\nu(\rho_{xy},t)},
\]
we obtain since  $r=\sqrt{t/2}\wedge R_2\leq \sqrt t$ and $ \exp(1)\leq 4 \leq 2^{\tilde N} $
\begin{align*}
	p_{t}(x,y)&
	\leq  2^{4{\tilde N_{xy}}}
	\tilde \Gamma_x(r/2)\tilde \Gamma_y(r/2)
	\frac{
		\left(1\vee {S^{-2}}{\left(\sqrt{t^2+\rho_{xy}^2S^2}-t\right)}\right)^{\frac{\tilde N_{xy}}{2}}}
	{\sqrt{m(B_{x}(r))m(B_{y}(r))}} 
	\euler^{-\Lambda (t-t\wedge R_2^2)-\zeta(\rho_{xy},t)}.
\end{align*}

The second part, which is the rest of the proof, is devoted to  estimate  the appearing terms which is done in three steps: We estimate the volume terms, the dimension term and the remaining error terms. 
We abbreviate $\tau=\sqrt t\wedge R_2$. Recall $  r=\sqrt{t/2}\wedge R_2 $ and let $ z\in \{x,y\} $.

In the first step, we estimate the volume terms $1/ m(B_{z}(r))\leq C_{z}(\tau)/m(B_{z}(\tau)) $ via volume doubling:
Since $ t\geq 16R_1^2$ and $ R_{2}\ge 4R_{1} $, we clearly have the trivial inequality $r= \sqrt{t/2} \wedge R_2\geq \tau/2\geq 2R_1$. The doubling property leads to
\begin{align*}
	\frac{1}{m(B_{z}(r))}
	=
	\frac{1}{m(B_{z}(\sqrt {t/2}\wedge R_2))}
	\leq 
	\frac{1}{m(B_{z}(\tau/2))}
	\leq 
	\frac{C_z(\tau)}{m(B_{z}(\tau))}
\end{align*}
with $	C_{z}(\tau)=2^{ {7}\tilde N^2_{z}(\tau)}\tilde \phi^{\tilde N_{z}(\tau)}
M_z(\tau)
^{2\vartheta(\tau)}   $. 

For the second step, we turn to the dimension terms. Note $\tau/2\leq r\leq \tau$. We have
\begin{align*}
\max_{[r,\tau]}\tilde	N_{z}&= \max_{R\in [r,\tau]}\Vert n_z\Vert_{[R/2,R] }\leq \Vert n_z\Vert_{[\tau/4,\tau] }= N_{z}(\tau)=:N.
\end{align*}
This yields the estimate $ \tilde N_{xy}=(\tilde N_{x}(\tau)+\tilde N_{y}(\tau))/2 \leq( N_{x}(\tau)+ N_{y}(\tau))/2=N_{xy}  $  needed for the correction term
\begin{align*}
\left(1\vee {S^{-2}}{\left(\sqrt{t^2+\rho_{xy}^2S^2}-t\right)}\right)^{\frac{\tilde N_{xy}}{2}}\leq \left(1\vee {S^{-2}}{\left(\sqrt{t^2+\rho_{xy}^2S^2}-t\right)}\right)^{\frac{ N_{xy}}{2}}.
\end{align*}

We are left to estimate $ 2^{{\tilde N} }\tilde \Gamma_{z}(r/2) \sqrt{C_{z}(\tau)}\leq A^{N}_{z}(\tau)\Psi_{z}(\tau) $ for each $ z\in \{x,y\} $ which is the third  and final  step. 
Here, we need some preliminary estimates next to the estimate on $ \tilde N $. The first is on the Sobolev constant, which we estimate in a similar fashion as $ \tilde N $ above using $\tau/2\leq r\leq \tau$
\begin{align*}
	\max_{[r,\tau]}\tilde	\phi_{z}&= \max_{R\in [r,\tau]}\Vert \phi_z\Vert_{[R/2,R] } \le  \Vert \phi_z\Vert_{[\tau/4,\tau] }=: \Phi .
\end{align*}
Next, we estimate  $ \vartheta(r):=\frac{1}{2}q_{z}(r)^{\eta(r)}$ with  $
q_z(r)=\frac{\tilde N+2}{\tilde N+4}$  and $ \eta(r)=\left\lfloor \frac12 \log_2\frac{r}{16 \tilde s_z(r)}\right\rfloor $. Clearly, we have
\begin{align*}
	\min\limits_{[r,\tau]}\eta
	=
	\min\limits_{R\in [r,\tau]}\left\lfloor \frac12 \log_2\frac{R}{16 \Vert s_z\Vert_{[ R/2,R]}}\right\rfloor
	\geq 
	\left\lfloor \frac12 \log_2\frac{\tau}{32\Vert s_x\Vert_{[\tau/4,\tau]}}\right\rfloor
	=\kappa_z(\tau)
\end{align*}
and from the estimate  $ \tilde N \leq N$ and since $q=\frac{\tilde N+2}{\tilde N+4}=1-\frac{2}{\tilde N+4}  $, we obtain
\begin{align*}
	\max_{[r,\tau]} \vartheta= \max_{[r,\tau]}\frac{1}{2}q_{z}^{\eta}
	\leq \max\limits_{R\in[r,\tau]}\left(\frac{\tilde N_z(R)+2}{\tilde N_z(R)+4}\right)^{\min\limits_{[r,\tau]}\eta}
	\leq \left(\frac{ N_z(\tau)+2}{\tilde N_z(\tau)+4}\right)^{\kappa_z(\tau)}=:\tilde\vartheta.
\end{align*}
With these preparations, we now  come to the final error term $ 2^{2\tilde N}\tilde \Gamma_{z}(r/2) \sqrt{C_{z}(\tau)}$. We start with the term $ C_{z}(r)=2^{ {7}\tilde N^2_{z}(r)}\tilde \phi_{z}(r)^{\tilde N_{z}(r)}
M_z(r)
^{2\vartheta(r)}    $ arising from volume doubling. This term can be estimated by
\begin{align*}
	\max_{[r,\tau]}C_{z} \leq 2^{ {7} N^2} \Phi^{ N}
	M_z(\tau)
	^{2\tilde{\vartheta}},
\end{align*}
where we used the short hands $ N=N_{z}(\tau) $, $ \Phi=\Vert \phi_z\Vert_{[\tau/4,\tau] }  $ and $\tilde{\vartheta}=(\frac{N+2}{N+4})^{\kappa_{z}(\tau)}$ from above. We
recall  
$ \tilde \Gamma_{z}(r/2)=
2^{{62}\tilde N^2}(\tilde\phi_{z}(r) C_{z}(r))^{\tilde N}
\left[\left(1+ (r/2)^2\Deg_z
\right)M_z(r/2)
\right]
^{\vartheta(r)} $ and  $r\leq \tau$ to estimate
\begin{align*}
	\tilde \Gamma_{z}(r/2)\sqrt{C_{z}(\tau)}&=2^{{62}\tilde N^2}(\tilde\phi_{z}(r)\cdot C_{z}(r))^{\tilde N}
	\left[\left(1+ \frac{r^2}{4}\Deg_z
	\right)M_z(r/2)
	\right]
	^{\vartheta(r)} \sqrt{C_{z}(\tau)}
	 \\
	&\leq 2^{{62} N^2}\left(\Phi \cdot 2^{ {7} N^2} \Phi^{ N}
	M_z(\tau)
	^{2\tilde{\vartheta}}\right)^{ N}
	\!\!
	\left[\left(1+ \tau^2\Deg_z
	\right)M_z(\tau)
	\right]
	^{\tilde{\vartheta}} 2^{{3} N^2} \Phi^{\frac{ N}{2}}
	M_z(\tau)
	^{\tilde{\vartheta}}
	\\
	&=
	2^{{65} N^2+ {7}N^{3}}\Phi^{\frac{3}{2}N+ N^{2}} \left(1+ \tau^2\Deg_z
	\right)
	^{\tilde{\vartheta}} M_z(\tau)^{2(N+1)\tilde{\vartheta}}\\
&	\leq 2^{ {40} N^{3}}\Phi^{2N^{2}}\left[\left(1+ \tau^2\Deg_z
	\right) M_z(\tau)\right]^{\Theta_{z}(\tau)},
\end{align*}
where we used $ N\ge 2 $ and $ \Theta_z(\tau)=3N_z(\tau)\left(\frac{N_z(\tau)+2}{N_z(\tau)+4}\right)^{ \kappa_z(\tau)} \ge2(N+1)\tilde \vartheta$. Hence, setting $ \Psi_z(\tau)=
\left[(1+\tau^2\Deg_z)
M_z(\tau)\right]^{\Theta_z(\tau)}$  and $ A_{z}^{N}(\tau)=2^{ {41}N_z(\tau)^3}\Vert\phi_z\Vert_{[\tau/4,\tau]}
^{2N_z(\tau)^2} $ reveals the estimate $  2^{2{\tilde N}}\tilde \Gamma_{z}(r/2)\sqrt{C_{z}(\tau)} \leq A^{N}_{z}(\tau)\Psi_{z}(\tau)$ which finishes the proof. 
\end{proof}

\section{Abstract uniform  on-diagonal bounds imply Sobolev}\label{section:heatsob}

{In this section} we show that the existence of a family of operators satisfying certain contractivity properties provides weak Sobolev inequalities. Then we use these weak Sobolev inequalities to derive Sobolev inequalities. Finally, we apply these results to obtain Sobolev inequalities from upper heat kernel bounds. The results and proofs are inspired by the corresponding results on manifolds \cite{SaloffC-01,Varopoulos-85}. Special attention is paid to the influence of the local geometry of the underlying graph on the behavior of the involved constants.
\\

\begin{lemma}[weak Sobolev]\label{lem:weakS}Let $C_1,C_2,n>0$, $0\leq r_1\leq r_2$,  
and $(Q_{r})_{r\in[r_1, r_2]}$ be a family of operators with domain containing $\cC_c(X)$ such that for  $f\in\cC( B(r_2))$ and all $r\in [ r_1, r_2]$
\[
\Vert Q_{r}f\Vert_\infty\leq C_1 r^{-n}\Vert f\Vert_1, \quad \Vert f-Q_{r}f\Vert_2\leq C_2 r\Vert |\nabla f|\Vert_2.
\]
Then,
\begin{equation*}
\sup_{\lambda>0}\lambda^{2(1+1/n)}m(f>\lambda)\leq 12  C_2^2
\left(C_1\vee \left(r_1^n \left\Vert\tfrac1m\right\Vert_{B(r_2)}\right)\right)^\frac{2}{n}
\left( \Vert|\nabla f|\Vert_2^2+\frac{1}{r_2^2}\Vert f\Vert_2^2\right)\Vert f\Vert_1^{2/n}.
\end{equation*}

\end{lemma}

\begin{proof}If $f=0$ in $B(r_2)$ there is nothing to prove, so assume $f\neq 0$ in $B(r_2)$.
If $\lambda>\Vert f\Vert_\infty$, then $\{f>\lambda\}=\emptyset$ and there is nothing to prove, so assume $\lambda \leq \Vert f\Vert_\infty$. We always have
\[
\lambda^2 m(f>\lambda)=\sum_X m \lambda^2\Eins_{\{f>\lambda\}}\leq \Vert f\Vert_2^2.
\]
Let $\mu>0$ to be chosen later. We distinguish between two cases for $\lambda\geq 0$.\medskip

First, if $\lambda\leq \mu C_1 r_2^{-n}\Vert f\Vert_1$, we get
\begin{multline*}
\lambda^{2(1+1/n)}m(f>\lambda)=\lambda^2m(f>\lambda)\lambda^{2/n}\leq  (\mu C_1)^{2/n} r_2^{-2}\Vert f\Vert_2^2\Vert f\Vert_1^{2/n}
\\\leq  (\mu C_1)^{2/n}\left( \Vert|\nabla f|\Vert_2^2+r_2^{-2}\Vert f\Vert_2^2\right)\Vert f\Vert_1^{2/n}.
\end{multline*}
Second, let $\lambda> \mu C_1 r_2^{-n}\Vert f\Vert_1$, and set \[
r_f:=\left(\mu\ C_1\Vert f\Vert_1\lambda^{-1}\right)^{1/n}.
\]
First, for the upper bound on $r_f$, note that the current lower bound of $\lambda$ yields
\[
r_f=
\left(\frac{\mu C_1\Vert f\Vert_1}{\lambda}\right)^{1/n}
\leq 
\left(\frac{\mu C_1\Vert f\Vert_1}{\mu C_1 r_2^{-n}\Vert f\Vert_1}\right)^{1/n}=r_2.
\]
Now, we estimate $r_f$ from below. 
Since
$\lambda\leq\Vert f\Vert_\infty$,
we obtain
\begin{align*}
r_f
=
\left(\frac{\mu C_1\Vert f\Vert_1}{\lambda}\right)^{1/n}
&\geq 
\left(\frac{\mu C_1\Vert f\Vert_1}{\Vert f\Vert_\infty}\right)^{1/n}.
\end{align*}
If we choose 
\[
\mu 
\geq \frac{r_1^n}{C_1} \frac{\Vert f\Vert_\infty}{\Vert f\Vert_1}, 
\]
then
$r_f\geq r_1$.
Hence, we can apply our assumptions to $r_f\in[ r_1,r_2]$. 
\medskip

For later use, we need to choose $\mu$ such that  $\vert Q_{r_f} f\vert <\lambda/2$ on $B(r_2)$: This is satisfied if we assume $\mu\geq 3$. Indeed, assuming the existence of $x\in B(r_2)$ with $\vert Q_{r_f} f\vert(x) \geq \lambda/2$, the definition of $r_f$ and our assumption on $\Vert Q_{r}f\Vert_\infty$ yield
\[
 \frac{3}{2}C_1\Vert f\Vert_1 r_f^{-n}
 \leq \frac{\mu}{2}C_1\Vert f\Vert_1 r_f^{-n}=\frac{\lambda}{2}\leq\vert Q_{r_f}f\vert(x)\leq \Vert Q_{r_f}f\Vert_\infty\leq C_1{r_f}^{-n}\Vert f\Vert_1,
\]
a contradiction.
\\
The restrictions on $\mu>0$ above  lead us to the choice 
\[
\mu= 3\left(1\vee \left(\frac{r_1^n}{ C_1} \frac{\Vert f\Vert_\infty}{\Vert f\Vert_1}\right)\right).
\]
The triangle inequality yields $$\{f> \lambda \}\subset \{\vert f-Q_{r}f\vert \geq\lambda/2\}\cup\{\vert Q_{r}f\vert\geq \lambda/2\}$$ for all $r\geq 0$ and $\lambda>0$. Hence, since $\vert Q_{r_f}f\vert <\lambda/2$ on $B (R_2)$, we obtain using our second assumption and the definition of $ r_{f} =\left(\frac{\mu\ C_1\Vert f\Vert_1}{\lambda}\right)^{1/n}$ and $ \mu=3\left(1\vee \left(\frac{r_1^n}{ C_1} \frac{\Vert f\Vert_\infty}{\Vert f\Vert_1}\right)\right) $
\begin{multline*}
m(f>\lambda)
=m\left(\vert f-Q_{r_f}f\vert \geq \lambda/2\right)
\leq \left(\frac{2}{\lambda}\right)^2\Vert f-Q_{r_f}f\Vert_2^2
\leq 
\left(\frac{2}{\lambda}\right)^2C_2^2r_f^2\Vert| \nabla f|\Vert_2^2
\\=
\left(\frac{2}{\lambda}\right)^2
C_2^2
\left(\left(
\mu C_1\Vert f\Vert_1\lambda^{-1}\right)^{1/n}\right)^2\Vert| \nabla f|\Vert_2^2
= \frac{4\cdot 3^{\frac{2}{n}}  C_2^2}{\lambda^{2(1+1/n)}}
\Vert |\nabla f|\Vert_2^2\big[
(C_1\Vert f\Vert_1)\vee \left(r_1^n \Vert f\Vert_\infty
\right)\big]^{\frac{2}{n}}.
\end{multline*}
In order to obtain the desired right-hand side, we use 
 \[\Vert f\Vert_\infty=\max_{B(r_2)}|f| 
\leq \sup_{B(r_2)}\frac1m \Vert f\Vert_1=\left\Vert\tfrac1m\right\Vert_{B(r_2)} \Vert f\Vert_1. \hfill\qedhere
\]
\end{proof}

In order to derive the Sobolev inequality from the lemma above, we provide a  version of the proof of \cite[Theorem~4.1]{BakryCLS-95} and \cite{Barlow-book,Delmotte-97}.

\begin{theorem}[Sobolev]\label{thm:SobM}
Let $C_1,C_2>0$, $n>2$, $0\leq r_1\leq r_2$, 
and $(Q_r)_{r\in[r_1,r_2]}$ a family of operators defined on $\cC_c(X)$ such that for all  $f\in\cC(X)$ with $\supp f\subset B(r_2)$ and all $r\in [r_1,r_2]$
\[
\Vert Q_rf\Vert_\infty\leq C_1 r^{-n}\Vert f\Vert_1, \quad \Vert f-Q_rf\Vert_2\leq C_2 r\Vert| \nabla f|\Vert_2.
\] Then, we have for all  $f\in\cC(X)$ with $\supp f\subset B(R_2)$
\begin{align*}
\Vert f\Vert_{\frac{2n}{n-2}}^2
\leq 2^{8+\frac{2n}{n-2}}C_2^2
\left(C_1\vee \left(r_1^n\left\Vert\tfrac1m\right\Vert_{B(r_2)}\right)\right)^\frac{2}{n}
\left(
\Vert |\nabla f|\Vert_2^2+\frac{1}{r_2^2}\Vert f\Vert_2^2\right).
\end{align*}
\end{theorem}
\begin{proof}
We choose the following partition of identity for $ k\in \mathbb{Z} $
\[
\Phi_k\colon [0,\infty)\to[0,\infty),\quad t\mapsto (t-2^k)_+\wedge 2^k
=
\begin{cases}
0&\colon t<2^k,\\
t-2^k &\colon 2^k\leq t<2^{k+1},\\
2^k&\colon t\geq 2^{k+1}.
\end{cases}
\]
Indeed, for $t\geq 0$, we can calculate by choosing $l\in \ZZ$  such that $2^l\leq t<2^{l+1}$
\[
\sum_{k\in \mathbb{Z}}\Phi_k(t)=t-2^l+\sum_{k=-\infty}^{l-1}2^k=t.
\]
If $t,s\geq0$, using $\sum_k a_k^2\leq (\sum_k \vert a_k\vert)^2$, this yields (w.l.o.g.~$t\geq s$, i.e., $ \phi_{k}(t)\ge \phi_{k}(s) $)
\[
(t-s)^2=\left(\sum_{k\in \mathbb{Z}}(\Phi_k(t)-\Phi_k(s))\right)^2
\geq \sum_{k\in \mathbb{Z}}(\Phi_k(t)-\Phi_k(s))^2.
\]
If $f\ge 0$, then we set $f_k:=\Phi_k\circ f$, $k\in \ZZ$. Hence, by the estimate above
\[
\sum_{k\in \mathbb{Z}}\Vert \vert \nabla f_k\vert\Vert_2^2\leq \Vert \vert \nabla f\vert\Vert_2^2.
\]
Moreover,
since $0\leq f_k\leq 2^k$ and $(\Phi_k)_k$ is a partition of the identity, we have 
\begin{align*}
\sum_{k\in \mathbb{Z}} \Vert f_k\Vert_2^2
\leq \Vert f\Vert_2^2.
\end{align*}
Hence, if we let
\[
W(f):=\Vert| \nabla f|\Vert_2^2+\frac{1}{r_2^2}\Vert f\Vert_2^2,
\]
then 
\[
W(f_k)
\leq W(f), \quad k\in\ZZ.
\]
 Abbreviate
\[C_M:=2^{4}C_2^2
\left(C_1\vee \left(r_1^n \left\Vert\tfrac1m\right\Vert_{B(r_2)}\right)\right)^\frac{2}{n},
\] 
and 
\[
N(f):=\sup_{k\in\ZZ} 2^k m(f\geq 2^k)^{\frac{1}{q}}, \qquad q=\frac{2n}{n-2}.
\]
\begin{claim}We have {for all $ f\ge 0 $}
\[
N(f)^{2}
\leq 
2^{q}C_M
W(f).
\] 
\end{claim}
We postpone the claim to the end of the proof and show how to derive the desired from the claim. Hence, we use
$ \{f_k\geq 2^k\}=\{f\geq 2^{k+1}\} $, the inequality  $\sum \vert a_k\vert^{q/2}\leq \left(\sum \vert a_k\vert\right)^{q/2}$ which is applicable since  $q/2\geq 1$ and  the claim to estimate
\begin{multline*}
\Vert f\Vert_q^q
=\sum_{X}\sum_{k\in\mathbb{Z}}\vert f\vert^q\Eins_{\{2^k\leq f< 2^{k+1}\}}
\leq \sum_{k\in\mathbb{Z}}2^{q(k+1)}\sum_{X}m \Eins_{\{2^k\leq f< 2^{k+1}\}}
\\
\leq \sum_{k\in\mathbb{Z}}2^{q(k+1)}m(f\geq 2^k)
=\sum_{k\in\mathbb{Z}}2^{q(k+2)}m(f\geq 2^{k+1})
=2^{2q}\sum_{k\in\mathbb{Z}}2^{qk}m(f_k\geq 2^{k})
\\
{\le}2^{2q}\sum_{k\in\mathbb{Z}}N(f_k)^q
\leq 2^{2q}\left(\sum_{k\in\mathbb{Z}}N(f_k)^2\right)^{q/2}
\leq 2^{2q}2^{q^2/2}C_M^{q/2}\left(\sum_{k\in\mathbb{Z}}
W(f_k)\right)^{q/2},
\end{multline*}
i.e.,
\[
\Vert f\Vert_q^2
\leq 2^{4+q}C_M\sum_{k\in\mathbb{Z}}W(f_k)\leq 2^{4+q}C_M W(f).
\]
For general $f$, decompose $f=f_+-f_-$ and obtain the above inequalities for $f_+$ and $f_-$. Since $ f_{+} $ and $ f_{-} $ are orthogonal in $\ell^2(X,m)$ and $ \||\nabla f_{+}|\|^{2} +\||\nabla f_{-}|\|^{2}\leq \||\nabla {f}|\|^{2} $, the claim follows for $f$.
\\

\noindent
\emph{Proof of the claim.} Denote $\tau=1+1/n$.
Since $f_k\leq 2^k$ and $ \supp f_{k} \subset \{f\geq 2^k\} $, we have {with $ q=2n/(n-2) $}
\[
\Vert f_k\Vert_1\leq 2^km(f\geq 2^k)
=2^{k\left(1-q\right)}(2^{k}m(f\geq 2^k)^{\frac{1}{q}})^{q}
\leq 
2^{k\left(1-q\right)}N(f)^{q}.
\]
{The weak Sobolev inequality,} Lemma~\ref{lem:weakS}, applied to $\lambda=2^k$ and $f_k$ yields together with {$ W(f_{k})\leq W(f) $ and} the above estimate 
\[
2^{k2\tau}m(f_k\geq 2^k)
\leq C_MW(f_k)\Vert f_k\Vert_1^{\frac2n}
\leq 
C_M2^{k\left(1-q\right)\frac2n}W(f)N(f)^{q\frac2n}.
\]
We have 
$ \{f\geq 2^{k+1}\}=\{f_k\geq 2^k\} $ and $ (k+1)q-k2\tau+k\left(1-q\right)\frac2n=q$, and hence,
\begin{multline*}
\left(2^{k+1}m(f\geq 2^{k+1})^{\frac{1}{q}}\right)^{q}
=2^{(k+1)q-k2\tau}2^{k2\tau} m(f_k\geq 2^k)
\\
\leq 
C_M2^{(k+1)q-k2\tau}2^{k\left(1-q\right)\frac2n}W(f)N(f)^{q\frac2n}
=
C_M2^{q
}W(f)N(f)^{q\frac2n}.
\end{multline*}
The definition of $N(f)$ yields \[
N(f)^{q}
\leq 
C_M2^{q}
W(f)N(f)^{q\frac2n}.
\] 
{Dividing by $N(f)^{2q/n}$}, using $q-q\frac2n=2$ yields the claim {and finishes the proof}.
\end{proof}

In the following we obtain the Sobolev inequality from  heat kernel bounds in balls depending on the local geometry of the graph.  The original idea of proof goes back to \cite{Varopoulos-85}. The argument is nowadays standard. We include the argument to track the constants and for the sake of completeness.

\begin{theorem}\label{lem:heattoS}
Let $0\leq r_1\leq r_2$, and assume for all $x,y\in  B(r_2)$
\[
p_{r^2}(x,y)\leq Cr^{-n}, \quad r\in [r_1,r_2].
\]
Then, for all {$f\in\cC( B(r_{2}))$}, we have 
\begin{align*}
\Vert f\Vert_{\frac{2n}{n-2}}^2
\leq 2^{8+\frac{2n}{n-2}}
\left(C\vee \left(r_1^{n} \left\Vert\tfrac1m\right\Vert_{B(r_2)}\right)\right)^\frac{2}{n}
\left(
\Vert \vert\nabla f\vert\Vert_2^2+\frac{1}{r_2^2}\Vert f\Vert_2^2\right).
\end{align*}

\end{theorem}
\begin{proof}
We need to check the assumptions of Theorem~\ref{thm:SobM} for $Q_r=P_{r^{2}}=e^{-r^2\Delta}$. 
Clearly, we have with $ B=B(r_2) $
\[
\Vert P_{r^2}\Vert_{\ell^1(B)\to\ell^\infty(B)}=\sup_{x,y\in B}p_{r^2}(x,y)\leq C r^{-n}, \quad r\in [r_1,r_2],
\]
what is the first assumption of Theorem~\ref{thm:SobM}.\\
Now, we check the second assumption. Fix $f\in \cC(X)$ with $\supp f\subset B$ and let $ t=r^{2} $. Then by self-adjointness of $P_t$ 
we have
\begin{align*}
\Vert f-P_tf\Vert_2=\Vert f\Vert_2^2+\Vert P_tf\Vert_2^2-2\langle P_tf,f\rangle 
=\Vert f\Vert_2^2+\Vert P_tf\Vert_2^2-2\Vert P_{t/2}f\Vert_2^2.
\end{align*}
By the contraction property of the heat semigroup, i.e., $\Vert P_tf\Vert_2^2\leq \Vert P_{t/2}f\Vert_2^2$, and the fundamental theorem of calculus we obtain
\begin{align*}
\Vert f-P_tf\Vert_2^2
\leq \Vert f\Vert_2^2-\Vert P_{t/2}f\Vert_2^2=-\int_0^{t/2}\frac{\drm }{\drm s}\Vert P_sf\Vert_2^2\drm s
=-2\int_0^{t/2}\left\langle \frac{\drm }{\drm s}P_sf,P_sf\right\rangle\drm s.
\end{align*}
Since $s\mapsto P_sf$ solves the heat equation $ -\frac{\drm }{\drm s}P_sf=\Delta P_sf $, Green's formula yields
\begin{align*}
\Vert f-P_tf\Vert_2^2
\leq 
2\int_0^{t/2}\langle \Delta P_sf,P_sf\rangle\drm s
=2\int_0^{t/2}\Vert |\nabla P_sf |\Vert_2^2\drm s,
\end{align*}
Analogously, the inequality
\[
\Vert |\nabla P_sf|\Vert_2^2\leq \Vert|\nabla f|\Vert_2^2\] 
can be seen by using the fundamental theorem, the heat equation, and Green's formula to obtain
\[
\Vert |\nabla P_sf|\Vert^2- \Vert|\nabla f|\Vert^2
=-\int_0^s \Vert \Delta P_sf\Vert_2^2\drm s
\leq 0.
\] 
Thus, we  get since $r^{2}=t$  and $ Q_{r}=P_{r^{2}} $
\begin{align*}
\Vert f-Q_{r}f\Vert_2^2=\Vert f-P_{t}f\Vert_2^2
\leq 
2\int_0^{t/2}\Vert |\nabla P_sf|\Vert_2^2\drm s
\leq t\Vert |\nabla f|\Vert_2^2=r^2\Vert |\nabla f|\Vert_2^2.
\end{align*}
The assumptions of {Theorem~\ref{thm:SobM}} are satisfied for $r\in[r_1,r_2]$ and the claim follows with $C_1=C$  and $C_2=1$.
\end{proof}

\section{{Gaussian} upper  bounds and {volume} doubling imply Sobolev}\label{section:GS}

In this section we show how to derive a Sobolev inequality  from Gaussian upper  bounds (G) and volume doubling (V). We pursue two lines of argument. The first more classical approach, Theorem~\ref{thm:heattosobolevgeneral}, going back to \cite{Varopoulos-85} collects all the remnants of the estimates within the   Sobolev constant. This is later used for the normalizing since in this case these remnants can be uniformly bounded. However, in the general unbounded case this is not feasible. In this case we mitigate the unbounded remains by choosing a variable dimension function to get a general preliminary Sobolev inequality Theorem~\ref{thm:heattosobolevgeneralimproving}. In this theorem, we still have a free parameter called $ \gamma $ which is then appropriately chosen in the next section.

Both approaches use a sequence of technical lemmas given below.   First we show how to bound the measure of a ball in terms of the measure of any other ball via chaining and the doubling condition. This is the only place in the paper where we use that $ \rho $ is a path metric. 
 \begin{lemma}[Comparing balls]\label{lem:ballcomparison}
 	Let $r\ge0$ such that $r\ge  8\| s(r/4)\|_{B_{o}(r)} $, $ d>0$, $\Phi\geq 1$  be constants, 
 	and assume $V_\Phi (d,r/4,r)$ in $B_{o}(r)$. Then, for all $x,y\in B(r)$, we have 
 	\[
 	m(B_x(r))
 	\leq
 	2^{18d}
 	\Phi^{9}
 	\ m(B_{y}(r)).
 	\]
 \end{lemma}
 \begin{proof} 
 	If $x, y\in B(r)$ such that $B_x(r/4)\cap B_y(r/4)\neq \emptyset$, then clearly $B_x(r/4)\subset B_y(r)$. By $V_\Phi(d,r/4,r)$, we have
 	\[
 	m(B_x(r))
 	\\
 	\leq 4^d\Phi\ m(B_x(r/4))\leq  4^d\Phi\ m(B_y(r)).
 	\]
 	
 	Let $x,y\in B(r)$ and 
 	denote by $p$ a path $x=z_0\sim z_1\sim \ldots \sim z_l= y$ realizing $\rho(x,y)$ which exists due to local finiteness cf. \cite[Chapter 12.2]{KellerLW-21}.  We construct a sequence of vertices $ (z_{k_i}) $ on $ p $  inductively and denote $ B_{i}=B_{z_{k_{i}}}(r/4) $. Set $k_0:=0$. For given $ k_{i} $ and $ B_{i} $, we choose the smallest index $ k =1,\ldots,l$ such that $ z_{k+1} $ is not in $\bigcup_{j=0}^{i} B_{j} $ but $ z_{k} $ is in $\bigcup_{j=0}^{i} B_{j} $ and set $ k_{i+1}={k} $. If there is no such $ z_{k+1} $ on the path we choose $ k_{i+1}=l $ and we set $ L=i+1 $. {Observe that $ z_{k}=z_{k_{i+1}}\in B_{i}=B_{z_{k_{i}}}(r/4) $ since if $ z_{k}\in B_{j} $ for $ j\neq i $ and using that the path $ z_0\sim\ldots\sim z_{l} $ realizes $ \rho(x,y) $, we get a contradiction 
 	\begin{align*}
 		\frac{r}{4}\ge \rho(z_{k},z_{k_{j}})= \rho(z_{k},z_{k_{i}})+ \rho(z_{k_{i}},z_{k_{j}})>\frac{r}{4}.
 	\end{align*}
 }Moreover, since $ z_{k+1}\not\in B_{i}= B_{z_{k_{i}}}(r/4) $ and {$ \rho(z_{k+1},z_{k_{i+1}})=\rho(z_{k},z_{k+1})\leq s $, we have, with $ s=\| s(r/4)\|_{B_{o}(r)} $, 
 	\begin{align*}
 		\rho ( z_{k_i}, z_{k_{i+1}})\ge \rho(z_{k_{i}},z_{k+1})-\rho(z_{k+1},z_{k_{i+1}})>\frac{r}{4}-s > 0
 	\end{align*}}
 	for $ i=0,\ldots,L-1 $. Thus, we have
 	\begin{align*}
 		r\ge \rho(x,y)=\sum_{i=0}^{L-1}\rho(z_{k_{i}},z_{k_{i+1}}) \ge (L-1)\left(\frac{r}{4}-s\right).
 	\end{align*}
 Since $ r\ge 8s $, we obtain
 	\begin{align*}
 		L\le \frac{5r-4s}{r-4s}= 5 + \frac{16s}{r-4s}\leq 9.
 	\end{align*} 	
 	At the same time, we have 
 	$ B_{i}\cap B_{i+1} \neq \emptyset$ since {$ z_{k_{i+1}}\in B_{i}$  as shown above $z_{k_{i+1}}\in  B_{i+1}=B_{z_{k_{i+1}}}(r/4) $} for $ i=0,\ldots, L $. 
 	Iterating the estimate in the beginning of the proof along $ (z_{k_{i}}) $ we obtain
 	\begin{align*}
 		m(B_{x}(r))\leq 4^{Ld}\Phi^{L}m(B_{y}(r))\leq 4^{9d}\Phi^{9}m(B_{y}(r))
 	\end{align*}
 	which yields the claim.
 \end{proof}

 Next, we need the notion of on-diagonal bounds which plays only a technical role in our considerations.


\begin{definition} Let $B\subset X$, $ r_{2}\ge r_{1}\ge 0 $, and $\Psi\colon B\times [r_1,r_2]\to (0,\infty)$. The \emph{on-diagonal estimate} $O_{\Psi}(r_1,r_2)$ is satisfied in $B$ if we have 
\[
p_{\rho^2}(x,x)\leq \frac{\Psi(x,\rho)^2}{m(B_x(\rho))},\quad x\in B, \rho\in[r_1,r_2].
\]
\end{definition}
Gaussian upper bounds $G_\Psi(n,r_1,r_2)$ in $B$ obviously imply $O_\Psi(r_1,r_2)$ in $B$ for all $[r_1,r_2]$. Indeed, the dimension $ n $ does not appear in the on-diagonal bounds and while $ G_\Psi(n,r_1,r_2)  $ is an assumption on all $ t=r^{2}\ge r_{1} ^{2}$, the on-diagonal bounds are only required for $ r \in[r_{1},r_{2}]$ (which is why no $ \sqrt{t}\wedge r_{2}=r\wedge r_{2} $ terms appear).

\begin{lemma}[Uniform on-diagonal bounds]\label{lem:heattosob}
Let 
$0\leq {4}r_1\leq r$,  $8\| s(r/4)\|_{B(r)}\leq r$, and constants $\Psi,\Phi\geq 1$, $d>2$ such that $O_\Psi(r_1,r)$ 
and $V_\Phi (d,r_1,r)$ hold in $B(r)$. 
Then  for $x,y\in B(r)$, $\sigma\in[r_1,r]$, we have
\[
p_{\sigma^2}(x,y)
\leq 
2^{18d}\Psi^{10}\Phi^{10}
\frac{r^{d}}{m(B(r))}\sigma^{-{d}}.
\]
\end{lemma}
\begin{remark}The proof of Lemma~\ref{lem:heattosob} requires the comparison of volumes of balls with same radius but different centers, Lemma~\ref{lem:ballcomparison}, which needs that the intrinsic metric is a path metric. Imposing $V_\Phi (d,r_1,2r)$ in $B(r)$ in Lemma~\ref{lem:heattosob} instead allows to drop this restriction on the metric.
\end{remark}
\begin{proof} Since $8\| s(r/4)\|_{B(r)}\leq r$, 
and $V_\Phi (d,r_1,r)$ hold in $B(r)=B_o(r)$, we use  Lemma~\ref{lem:ballcomparison} to infer
\[
m(B(r))\leq 2^{18d}\Phi^9\ m(B_x(r)), \quad x\in B(r).
\]
For any $\sigma\in[r_1,r]$ and $x\in B(r)$, we obtain from  $O_{\Psi}(r_1,r)$  and $V_\Phi (d, r_1,r)$ in $x$
\begin{align*}
p_{\sigma^{2}}(x,x)\leq \frac{\Psi^2}{m(B_x(\sigma))}
\leq 
\frac{\Psi^2\Phi }{m(B_x(r))}\left(\frac{r}{\sigma}\right)^{d}
\leq 
\frac{2^{18d}\Psi^{10}\Phi^{10}}{m(B(r))} \left(\frac{r}{\sigma}\right)^{{d}}
.
\end{align*}
For $\sigma\in[r_1,r]$  and $x,y\in B(r)$, we infer
\[
p_{\sigma^2}(x,y)\leq \sqrt{p_{\sigma^2}(x,x)p_{\sigma^2}(y,y)}
\leq 
\frac{2^{18d}\Psi^{10}\Phi^{10}}{m(B(r))}\left(\frac{r}{\sigma}\right)^{{d}},
\]
where the first inequality follows directly from the semigroup identity and Cauchy-Schwarz inequality.
Thus, $\sigma\leq r$ implies the claim for all $n\geq d$.
\end{proof}


\begin{lemma}[On-diagonal bounds an volume doubling imply Sobolev]\label{lem:heattosob3}
Let 
$0\leq {4} r_1\leq r$, $8\| s(r/4)\|_{B(r)}\leq r$ and constants $\Psi,\Phi\geq 1$, $d>2$, and assume $O_\Psi(r_1,r)$ and 
 $V_\Phi (d,r_1,r)$ hold in $B(r)$. 
Then  for all $n\geq d$ we have $ S_{\tilde \phi}(n,r)$ in $o$, where  $\tilde \phi =\tilde \phi(r_1,r) $ is given by
\[
\tilde \phi(r_1,r)=2^{44+\frac{2n}{n-2}}
\left[\Psi^{10}\Phi^{10}\vee r_1^n\frac{m(B(r))}{r^{n}}\left\Vert\tfrac1m\right\Vert_{B(r)}\right]^\frac{2}{{n}}.
\]
\end{lemma}

\begin{proof}
Lemma~\ref{lem:heattosob} yields for all $n\geq d$, $\sigma\in[r_1,r]$  and $x,y\in B(r)$
\[
p_{\sigma^2}(x,y)
\leq 
C
\sigma^{-{n}}, \quad \text{where}\quad C=2^{18n}\Psi^{10}\Phi^{10}\frac{r^{n}}{m(B(r))}.
\]
Thus, if $n\geq d$,
Theorem~\ref{lem:heattoS} yields
for all $f\in\cC( B(r))$
\begin{align*}
\Vert f\Vert_{\frac{2{n}}{{n}-2}}^2
&
\leq 2^{8+\frac{2n}{n-2}}\left(C\vee \left(r_1^{n} \left\Vert\tfrac1m\right\Vert_{B(r)}\right)\right)^{{\frac{2}{n}}}\left(
\Vert \vert\nabla f\vert\Vert_2^2+\frac{1}{r^2}\Vert f\Vert_2^2\right)
\end{align*}
which yields the result since   $2^{8+\frac{2n}{n-2}}(C\vee (r_1^{n} \left\Vert\tfrac1m\right\Vert_{B(r)}))^{{\frac{2}{n}}}\leq \tilde\phi(r_{1},{r})  $.
\end{proof}

From now, we let the error terms be given by functions rather than constants. Let for $0\leq r_1\leq r$
\begin{align*}
	  \Phi\colon X\times[r_1,r]\times[r_1,r]\to[1,\infty),\quad& {(x,\sigma,\tau)\mapsto \Phi_{x}^{\tau}(\sigma)}\\
	  \Psi\colon X\times [r_1,r]\to [1,\infty),\quad& {(x,\sigma)\mapsto \Psi_{x}(\sigma)}
\end{align*}
be given. Set
\[
Q(r,s)=B(s)\times [r,s].
\]
Without adapting the dimensions we end up with the following.

\begin{theorem}[Sobolev -- fixed dimension]\label{thm:heattosobolevgeneral}Assume $ n>2 $ is a constant
$0\leq {4} r_1\leq r_2$, $8\| s(r_2/4)\|_{B(r_2)}\leq r_2$, and that $ O_\Psi(r_1,r_2)$ and $V_\Phi (n,r_1,r_2)$ hold in $B(r_2)=B_{o}(r_{2})$. 
Then we have $S_\phi(n, {4}r_1,r_2)$ in $o$, where
\begin{align*}
\phi(r)&=2^{44+\frac{2n}{n-2}}\left[ \Vert \Psi\Vert_{Q(r_1,r)}^{10}\Vert\Phi^r\Vert_{Q(r_1,r)}^{10}\vee \left( r_1^{n}\frac{m(B(r))}{r^{n}}\left\Vert\tfrac1m\right\Vert_{B(r)}\right)\right]^{\frac{2}{ 
n}}.
\end{align*}
\end{theorem}

\begin{proof}
For  $r\in[r_1,r_2]$,  Lemma~\ref{lem:heattosob3} implies $S_{\tilde \phi}(n,r)$ in $o$ with $ \tilde\phi =\tilde{\phi}(r_{1},r)$.
\end{proof}

From Theorem~\ref{thm:heattosobolevgeneral} we immediately get  (G) \& (V) $ \Rightarrow $ (S) with the respective Sobolev constant above. This is acceptable if the norms of $ \Psi $, $ \Phi $ and $ 1/m $ stay bounded. For graphs with unbounded geometry this cannot be expected beyond the normalizing measure.

However, as one can see, a large dimension $ n $ potentially mitigates  large norms  of $ \Psi $, $ \Phi $ and $ 1/m $ as it enters as a large root in the Sobolev constant. We pursue this strategy in  the following.  First, in Lemma~\ref{lem:variabledim1}, we estimate the dimension such that the Sobolev constant $ \phi  $ stays uniformly bounded. Secondly, in Theorem~\ref{thm:heattosobolevgeneralimproving}, we choose radii $ r_{1} $ in dependence of $ r_{2} $ such that the Sobolev dimension $ n(r_{2}) $ converges to $ d $ as $ r_{2}\to \infty $.

We now also allow for  a  variable dimension  
\[
d\colon X\times [r_1,r]\to (2,\infty)
\]
in the volume doubling property.

\begin{lemma}[Choosing a dimension]\label{lem:variabledim1}Let 
$0\leq {4}r'\leq r$, $8\| s(r/4)\|_{B(r)}\leq r$, let $\Psi,\Phi\geq 1$, $d>2$ be constants, and assume $O_\Psi(r',r)$  and 
 $V_\Phi (d,r',r)$ hold in $B(r)=B_{o}(r)$. 
Then for all $\gamma\geq 1$ we have $ S_{\hat \phi}(n,r)$ in $o$, where 
\[
\hat{\phi}=\hat{\phi}(d)=2^{47+\frac{2d}{d-2}}
\gamma^{\frac{2}{d}}
\]
and  the dimension $ n $ can be chosen such that
\[
n\geq d\vee
\ln
\left(1\vee
\frac{\Psi^{10}\Phi^{10}}{\gamma}\vee 
\left[\frac{m(B(r))}{\gamma}\|\tfrac1m\|_{B(r)}
\right]^{\frac{1}{\ln ({r}/{r'})}}
\right).
\]
\end{lemma}

\begin{proof}
Since $0\leq {4}r'\leq r$, Lemma~\ref{lem:heattosob3}
yields for all $n\geq d$ the Sobolev inequality $ S_{\tilde{\phi}}(n,r)$ in $o$ with
\[
\tilde{\phi}=\tilde{\phi}(r',r)=2^{44+\frac{2n}{n-2}}
\left[\frac{\Psi^{10}\Phi^{10}}\gamma\vee (r')^n\frac{m(B(r))}{\gamma r^{n}}\left\Vert\tfrac1m\right\Vert_{B(r)}\right]^\frac{2}{{n}}\gamma^\frac{2}{{n}},
\]
where we snuck in  $\gamma\geq 1$. The second entry of the maximum can be estimated by $ 1 $, i.e.,
\[
(r')^n\frac{ m(B(r))}{\gamma r^n}\left\Vert\tfrac1m\right\Vert_{B(r)}\leq 1
\]
if and only if
\[
\ln \left(\frac{m(B(r))}\gamma\left\Vert\tfrac1m\right\Vert_{B(r)}\right)^{\frac{1}{{\ln ({r}/{r'})}}}\leq n.
\]
Hence, under this condition,
\[
\tilde\phi(r',r)\leq 2^{44+\frac{2n}{n-2}}
\left[1\vee \frac{\Psi^{10}\Phi^{10}}\gamma\right]^\frac{2}{{n}}\gamma^\frac{2}{{n}}.
\]
Note that if $t\geq 1$, then we have for all  $n\geq \ln t$ the estimate $t^{1/n}\leq t^{1/\ln t}=\euler$. Hence, if we further require
\[
n\geq \ln\left(\frac{\Psi^{10}\Phi^{10}}\gamma\vee 1\right)
\]
and using $e^2\leq 2^3$, we obtain
\[
\tilde{\phi}(r',r)\leq 2^{44+\frac{2n}{n-2}}
\euler^2\gamma^\frac{2}{{n}}\leq 2^{47+\frac{2n}{n-2}}
\gamma^\frac{2}{{n}}\leq 2^{47+\frac{2d}{d-2}}
\gamma^{\frac{2}{d}}=\hat{\phi}(d)
\]
since the function $ n\mapsto 2n/(n-2) $ is decreasing and $ n\ge d $.
Putting the conditions on $n$ together yields the claim.
\end{proof}

Now we are in the position to choose specific radii. The parameter $p$ below is needed for the correct order of convergence and the parameter $ \gamma $  to adjust the error terms later.

\begin{theorem}[Sobolev -- variable dimension]\label{thm:heattosobolevgeneralimproving}
Let 
$0\leq { 4} R_1\leq R_2$, $8\| s(r/4)\|_{B(r)}\leq r$ for all $r\in[R_1,R_2]$, and assume $ G_{\Psi}(N,R_1,R_2)$ and $V_{\Phi} (d,R_1,R_2)$ hold in $B(R_2)=B_{o}(R_{2})$. 
Then for all functions $\gamma\colon [4R_1,R_2]\to [1,\infty)$ we have $S_{\phi}(n,{ 4} R_1,R_2)$ in $o$, where
\begin{align*}
	\phi(r)=2^{49+\frac{2 D(r)}{D(r)-2}}\gamma(r)^{\frac{2}{D(r)}}
\end{align*}
with    $D(r)=\|d(r)\|_{B(r)}$ and the dimension $ n $ can be chosen
\begin{align*}
n(r)\geq D(r)\vee \ln\left[1\vee
\frac{\Vert\Psi\Vert_{Q(r',r)}^{10}\Vert\Phi^r\Vert_{Q(r',r)}^{10}}{\gamma(r)}\vee \left(\frac{m(B(r))\Vert \tfrac1m\Vert_{B(r)}}{\gamma(r)}\right)^{\frac{1}{\ln(r/r')}}
\right]
\end{align*}
with 
\[
 r'=
\begin{cases}
r/ {4} &\colon r\in[{4}R_{1},\exp({4\vee 4 R_1})),\\
(\ln r)^{p(r)} {/4}&\colon r\geq\exp({4\vee 4 R_1}),
\end{cases} \qquad
p(r)=\frac{2}{\ln\left(1+\frac{2}{D(r)+2}\right)}
.
\]
\end{theorem}

\begin{proof}Gaussian bounds $G_\Psi(N,R_1,R_2)$ in $B(R_2)$ yields on-diagonal bounds $O_\Psi(r',r)$ in $B(r)$ if $R_1\leq {r}'\leq r\leq R_2$. Similarly, $V_\Phi (d,R_1,R_2)$ in $B(R_2)$ yields $V_\Phi (D,r',r)$ in $B(r)$ if $R_1\leq {r}'\leq r'\leq r\leq R_2$. With these observations the 
result follows immediately from  Lemma~\ref{lem:variabledim1} applied to the interval $[r',r]$ if we can show $R_1\leq {r}'$  and  ${4} r'\leq r$ for all $r\in[R_1,R_2]$. 
\\
If $r\in[{4}R_{1},\exp({4\vee 4 R_1}))$, then  $R_1\leq r/{4}=r'$  and  ${4}r'= r$.\\
Now, we turn to the case $r\geq \exp({4\vee 4 R_1})$. Since $D(r)\geq 2$, we have the inequality $\ln(1+2/(D+2))\leq \ln 2< 2$ and hence
\[
{p(r)}=\frac{2}{\ln\left(1+\frac{2}{D(r)+2}\right)}>1.
\]
Since $r\geq \exp({4\vee 4 R_1})\geq \exp(1)$, we have $\ln r\geq 1$, such that $p(r)>1$ particularly implies $\ln r\leq (\ln r)^{p(r)}$. Therefore,
\[
R_1\leq 1\vee R_1\leq {\tfrac{1}{4}}\ln r\leq {\tfrac{1}{4}}(\ln r)^{p(r)}=r'.
\]
Further, since $1/p(r)<1$ and all roots are larger than the logarithmic function, we get $\ln r\leq r^{1/p(r)}$ and, thus,
\[
{4}r'=(\ln r)^{p(r)}\leq r,
\]
which finishes the proof.
\end{proof}

\section{Proof of the main theorems involving local regularity}\label{section:remaining}

We are now ready to prove the main theorems. First, we provide the proof of Theorem~\ref{thm:equivnormalized} for the normalizing measure. Afterwards, we will deal with the general case for which we first show a result where the bounds depend on  vertex degrees and reciprocals of measures, Theorem~\ref{thm:general}. To reduce the bounds to depend on the vertex degree alone, we incorporate  the local regularity property (L) to show Theorem~\ref{thm:generaldeg1}. Finally,  Theorem~\ref{thm:counting} is an immediate corollary of Theorem~\ref{thm:generaldeg1}.
 
\begin{proof}[Proof of Theorem~\ref{thm:equivnormalized}]$(i)$ We have to show that $S_{\phi}(n,R_1,R_2)$ in $B$ implies properties $V_{\Phi}(n,R_1,R_2)$ and $G_{\Psi}(n,4R_1,R_2)$ in $B$ for appropriate $\Phi,\Psi>0$. \\
Corollary~\ref{cor:adapteddoubling}  yields $V_{\Phi}(n,R_1,R_2)$ with $\Phi=2^{ {10}n^2}\phi^{2n}$.
\\
From Theorem~\ref{thm:main1sturm} we infer $G_\Psi(n,4R_1,R_2)$ with 
\[
\Psi_z(\tau)=2^{ {41}n^3}\phi^{2n^2}
 \left[(1+\tau^2)
\frac{m(B_z(\tau))}{m(z)}\right]^{\Theta_z(\tau)}
\quad\text{
and}\quad \Theta_z(\tau)=3n\left(\frac{n+2}{n+4}\right)^{\kappa(\tau)} 
\]
with $ \kappa(\tau)=\lfloor \tfrac{1}{2}\log_{2} (\tau/32) \rfloor$ as the jump size satisfies $ S=1 $ for the normalizing measure and the combinatorial graph distance. We are left to show that $ \Psi $ is uniformly bounded in $ z $ and $ \tau $. 
By Lemma~\ref{lemma:gammanormalized} we have since $ \mathrm{Deg}=1 $ for the normalizing measure
\[
\frac{m(B_z(\tau))}{m(z)}\leq  \left[2\phi\left(1+\tau^2\right)\right]^{\frac{n}{2}}.
\]
Now, one proceeds as in the proof of Corollary~\ref{cor:adapteddoubling} to show that the right-hand side multiplied by $ (1+\tau^{2}) $ and raised to the power $\Theta$ is bounded.\\
\noindent
$(ii)$ We have to show that conditions  $V_{\Phi}(n,R_1,R_2)$ and $G_{\Psi}(n,R_1,R_2)$ imply $S_{\phi}(n,4R_1,R_2)$  for appropriate $\phi>0$. By Theorem~\ref{thm:heattosobolevgeneral} we have  $S_{ \phi}(n,4R_1,R_2)$ in $o$ with 
\[
\phi(r)=2^{44+\frac{2n}{n-2}}\left(\Psi^{10}\Phi^{10}\vee R_1^{n}\frac{m(B_o(r))}{r^{n}}\left\|\tfrac{1}{m}\right\|_{B_o(r)}\right)^{\frac{2}{ n}}.
\]
Since distance balls are finite and $m(x)=\deg(x)$, there exists $x\in B_o(r)$ such that 
\[
\left\|\tfrac{1}{m}\right\|_{B_o(r)}=\frac{1}{m(x)}=\frac{1}{\deg(x)}.
\]
Lemma~\ref{lem:ballcomparison} yields $m(B_o(r))\leq 2^{18n}\Phi^9m(B_x(r))$. Together with volume doubling, we obtain
\begin{align*}
	R_1^{n}\frac{m(B_{o}(r))}{r^{n}}\left\|\tfrac{1}{m}\right\|_{B_o(r)} 
	=\frac{R_1^{n}}{\deg(x)}\frac{m(B_{o}(r))}{r^{n}}
	\leq \frac{2^{18n}\Phi^9 R_1^{n}}{\deg(x)}\frac{m(B_{x}(r))}{r^{n}}
	&\leq 2^{18n}\Phi^{10} \frac{m(B_x(R_1))}{\deg(x)}
\end{align*}
which is bounded by assumption. Hence, $\phi$ is bounded above and the claim follows.
\end{proof}

In the case of normalizing measure, we assume that the volume of small balls and the vertex degree are comparable to obtain uniform constants. If we consider general measures, correction terms in $(G)$, $(V)$, and $(S)$ also depend on the vertex degree, which may be unbounded. Hence, we cannot hope for an analogous statement as for the normalizing measure. An additional upper bound on the vertex degree only leads to a minor generalization, and the correction terms become unbounded if the vertex degree grows. In order to include unbounded vertex degree into our results, we employ the results from the preceding sections. They lead to the following general version of the equivalence of Sobolev inequalities and heat kernel bounds involving varying dimensions.
\\

Let $R_{2}\ge  {4}R_1\geq0$. For a dimension function $n\colon X\times [R_1,R_2]\to (2,\infty)$ and $ x\in X $ and $ r\in [R_{1},R_2] $, we let the supremum over annuli be given as
\begin{align*}
	N_x(r)=\Vert n_x\Vert_{[r/4,r]}.
\end{align*}
Furthermore, the volume doubling  and the Gaussian correction term are given as
\[
\Phi_x^R(r)=\left(r^{N_{x}(R)}\frac{M_x(R)}{R^{N_{x}(R)}}
\right)^{\theta_x^N(r,R)},
\qquad \Psi_x(r)=
\left[(1+r^2\Deg_x)
M_x(r)
\right]^{3N_{x}(r)\theta_x^{N}(r/16,r)}
\]
with $ M_x(r)={m(B_x(r))}/{m(x)} $ and the exponent
\[
 {\theta^N_x(r,R)=\left(\frac{N_x(R)+2}{N_x(R)+4}\right)^{\eta(r)},\quad \eta(r)=\eta_{x}(r)=\left\lfloor \frac12\log_2\frac{r}{2 {\|s_x\|_{[r/2,r]}}}\right\rfloor,}
\]
where $  s_{x}(r) $ is the jump size for leaving the ball $ B_{x}(r) $.
For a constant $ \phi\ge 1$,  let $$ 	A_x'(r)
=2^{ {41}N_x(r)^3}\phi^{2N_x(r)^2}. $$
For the dimension function $ n $, we consider  the supremum over the space-time cylinder 
$ Q(r',r)=B(r)\times [r',r] $ and let
\begin{align*}
	N(r)=\Vert n\Vert_{Q(r',r)},
\end{align*}
\[
r'=
\begin{cases}
	r/ {4} &\colon r\in[{4}R_{1},\exp({4\vee 4 R_1})),\\
	(\ln r)^{p(r)} {/4}&\colon r\geq\exp({4\vee 4 R_1}),
\end{cases} 
\qquad 
p(r)=\frac{2}{\ln\left(1+\frac{2}{N(r)+2}\right)}.
\]
Next, for 
$\gamma\colon X\times [R_1,R_2]\to (0,\infty)$, we define a function $ \phi^{\gamma} $, which will play the role of the Sobolev constant, by
\begin{equation*}
\phi^{\gamma}(r)=2^{{49}+\frac{2N(r)}{N(r)-2}}\gamma(r)^{\frac{2}{N(r)}}.
\end{equation*}
Finally, we choose a dimension function $ n^{\gamma}$ which satisfies the following inequality
\begin{align*}
n^{\gamma}(r)\geq N(r)\vee \ln\left[1\vee
\frac{\Vert A'\Psi\Vert_{Q(r',r)}^{10}\Vert A'\Phi^r\Vert_{Q(r',r)}^{10}}{\gamma(r)}\vee \left(\frac{m(B(r))\Vert \tfrac1m\Vert_{B(r)}}{\gamma(r)}\right)^{\frac{1}{\ln(r/r')}}
\right].
\end{align*}

With this notation the main result in its most general form is now just a consequence of Theorems~\ref{thm:adapteddoubling}, \ref{thm:main1sturm}, and \ref{thm:heattosobolevgeneralimproving}.

\begin{theorem}[Most general case]\label{thm:general}Let 
$\diam(X)/2\geq R_2\geq  { 4}R_1\geq 0$, $\gamma\colon X\times [R_1,R_2]\to (0,\infty)$, and  $ \phi\ge 1$ a constant. 
\begin{enumerate}[(i)]
\item If  $S_\phi (n,R_1,R_2)$ holds in $B\subset X$,  {$R_1\geq 2\|s(0)\|_{B}$,} $r\geq 1024  \|s(r)\|_B$, $r\in[R_1,R_2]$,  then
$B$ satisfies $ V_{A'\Phi}(N,R_1,R_2)$ and $ G_{A'\Psi}(N,4R_1,R_2)$.
\item 
Assume $V_{A'\Phi} (N,R_1,R_2)$ and $ G_{A'\Psi}(N,R_1,R_2)$ hold in $B_o(R_2)$ and $2\| s(r)\|_{B_o(4r)}\leq r$, $r\in[R_1/4,R_2/4]$. Then the property $ S_{\phi^{\gamma}}(n^{\gamma},4 R_1,R_2)$ holds in $o$.
\end{enumerate}
\end{theorem}

In the case of the normalizing measure, we assumed additionally that measures of small balls are comparable to the vertex degree.
To obtain the Theorems~\ref{thm:counting} and Theorem~\ref{thm:generaldeg1} presented in the introduction, we will incorporate the the local regularity condition (L) into Theorem~\ref{thm:general}. The resulting correction functions and dimensions appearing in $(G)$, $(V)$, and $(S)$ the depend only on the vertex degree.
We will first prove Theorem~\ref{thm:generaldeg1} before we reduce it to the counting measure case Theorem~\ref{thm:counting}.

\begin{proof}[Proof of Theorem~\ref{thm:generaldeg1}] In the following, we abbreviate $
D_x(r):=(1+r^2\Deg_x).$

\noindent $ (i) $  Theorem~\ref{thm:general}~$ (i) $ yields $ V_{A'\tilde \Phi}(n,R_1,R_2)$ and $G_{A'\tilde \Psi}(N,4R_1,R_2)$ in $B$, 
where
\[
\tilde \Phi_x^R(r)=\left[r^{n_x(R)}\frac{M_x(R)}{R^{n_x(R)}}\right]^{\theta^n_x(r,R)},\quad\text{}\quad
\tilde \Psi_x(r)=
\left[D_x(r)
M_x(r)
\right]^{3N_{x}(r)\theta_x^{N}(r/16,r)}
\]
 $ 	A_x'(r)
=2^{ {41}N_x(r)^3}\phi^{2N_x(r)^2} $ and $ M_{x}(r) = m(B_{x}(r))/m(x)$. Lemma~\ref{lemma:gammanormalized} implies $L_\phi(n,R_1,R_2)$ in $B$, i.e., for all $r\in[R_1,R_2]$, $x\in B$,
\[
\frac{M_{x}(r)}{r^{n_x(r)}}
\leq  \left[\frac{2\phi}{r^2}D_x(r)\right]^{\frac{n_x(r)}{2}} \quad\text{or equivalently}\quad 
M_{x}(r)
\leq  \left[2\phi D_x(r)\right]^{\frac{n_x(r)}{2}}.
\]
The first estimate, $r\leq R$, $\phi,D_x\geq1$, and $\theta^N\leq 1$ yield
\begin{equation*}
\tilde \Phi_x^R(r)\leq 
\left[2\phi\frac{r^2}{R^2}D_x(R)\right]^{\frac{n_x(R)}{2}\theta^n_x(r,R)}
\hspace{-.1cm}\leq (2\phi)^{3N_x(R)^2}D_x(R)^{3N_x(R)^2\theta^N_x(r,R)}=(2\phi)^{3N_x(R)^2}\Phi_x^R(r).
\end{equation*}
Since $n_x(r)>2$, the second estimate, $ M_{x}(r)
\leq  \left[2\phi D_x(r)\right]^{{n_x(r)}/{2}} $, implies
\begin{equation*}
\Psi_x(r)\leq 
\left[D_x(r)
\left[2\phi D_x(r)\right]^{\frac{n_x(r)}{2}}
\right]^{3N_{x}(r)\theta_x^{N}(r/16,r)} \leq (2\phi)^{3N_x(r)^2}\Phi_x^r(r/16).
\end{equation*}
Hence, we obtain $V_{A\Phi}(N,R_1,R_2)$ and $G_{A\Psi}(N,4R_1,R_2)$ with $A_{x}'(r)=2^{ {41}N_x(r)^3}\phi^{2N_x(r)^2}$,
\[
(2\phi)^{3N_x(r)^2}A_x'(r)=(2\phi)^{3N_x(r)^2}2^{ {41}N_x(r)^3}\phi^{2N_x(r)^2}\leq 2^{ {43}N_x(r)^3}\phi^{8N_x(r)^2}=A_{x}(r).
\]

\noindent $ (ii) $ Abbreviate $\tilde N:= \Vert N\Vert_{Q(r',r)}$.  The properties
  $L_\phi(N,R_1,R_2)$, $V_{A\Phi}(N,R_1,R_2)$, and $G_{A\Psi}(N,R_1,R_2)$ in $B(R_2)$ yield  $L_\phi(\tilde N,R_1,R_2)$, $V_{A\Phi}(\tilde N,R_1,R_2)$, and $G_{A\Psi}(\tilde N,R_1,R_2)$ in $B(R_2)$. Hence, for any $\gamma\colon[R_1,R_2]\to[1,\infty)$, Theorem~\ref{thm:general}~(ii) (together with Lemma~\ref{lem:sobdim}) yields $S_{\tilde \phi}(\tilde n, 4R_1,R_2)$ in $o$ for $ \tilde \phi\ge \phi^{\gamma} $ and $ \tilde n\ge n^{\gamma} $ where 
  
\[
\phi^{\gamma}(r)=2^{47+\frac{2\tilde N}{\tilde N-2}}\gamma(r)^{\frac{2}{\tilde N}},
\quad
n^{\gamma}(r)= \tilde N\vee \ln\left[1\vee
 \frac{T_1}{\gamma(r)}\vee \left(\frac{T_2}{\gamma(r)}\right)^{\frac{1}{\ln(r/r')}}
\right]
\]
\[
T_1:=\Vert A\Psi\Vert_{Q(r',r)}^{10}\Vert A\Phi^r\Vert_{Q(r',r)}^{10},\qquad T_2:=m(B(r))\|\tfrac1m\|_{B(r)},
\] and 
\[
r'=
\begin{cases}
	r/ {4} &\colon r\in[{4}R_{1},\exp({4\vee 4 R_1})),\\
	(\ln r)^{p(r)} {/4}&\colon r\geq\exp({4\vee 4 R_1}),
\end{cases} 
\qquad 
p(r)=\frac{2}{\ln\left(1+\frac{2}{\tilde N+2}\right)}.
\]
In order to bound $n^{\gamma}$ from above, we need to estimate $T_1$ and $T_2$ from above and choose  $\gamma$ appropriately. We start with $T_1$ and abbreviate  $ \tilde  s := \| s\|_{ Q(r',r)}$ and 
$$ \tilde \theta:= \left(\frac{\tilde N+2}{\tilde N+4}\right)^{\left\lfloor \frac12\log_2\frac{r'}{32\tilde s}\right\rfloor},$$ and $\tilde D:=\Vert D(r)\Vert_{B(r)}$. Then,
\[
\Vert \Phi^r\Vert_{Q(r',r)}
=\sup_{(x,t)\in Q(r',r)} 
 D_x(r)^{3N_x(r)^2\theta^N_x(t,r)}
\leq 
\tilde  D^{3\tilde N^2\tilde \theta}
\]
and analogously
\begin{align*}
	\Vert \Psi\Vert_{Q(r',r)}
	=\sup_{(x,t)\in Q(r',r)}  \Phi^r_{x}(r/16)
	\leq
	\tilde D^{3\tilde N^2\tilde \theta}.
\end{align*}
We conclude with $ \tilde A=  \Vert A\Vert_{Q_o(r',r)}$
\[
T_1
\leq 
\Vert A\Vert_{Q( r',r)}^2\Vert \Psi\Vert_{Q(r',r)}\Vert \Phi^r\Vert_{Q(r',r)}
\leq\tilde A^2\tilde D^{6\tilde N^2\tilde \theta}.
\]
Now, we estimate $T_2$. First, since distance balls are finite by assumption, for all $B(r)\subset X$, $B(r)\neq X$, there exists  $x\in B(r)$ such that 
\[
\Vert \tfrac1m\Vert_{B(r)}=m(x)^{-1}.
\] 
As $V_{A\Phi}(N,R_1,R_2)$ implies $V_{\tilde A\tilde \Phi}(\tilde N,R_1,R_2)$ with $ \tilde \Phi=\|\Phi^{r}\|_{Q(r',r)} $, Lemma~\ref{lem:ballcomparison} gives
\[
m(B_o(r))
\leq
2^{18\tilde N}
\tilde A^9\tilde\Phi^{9}
\ m(B_{x}(r)).
\]
We infer from $L_\phi(\tilde N,R_1,R_2)$ as in (i) the estimate 
\[
\frac{m(B_{x}(r))}{m(x)}=M_{x}(r)
\leq [2\phi
 D_{x}(r)]^{\frac{\tilde N}{2}}
\leq (2\phi)^{\frac{\tilde N}{2}}
\tilde  D^{\frac{\tilde N}{2}}.\] 
Hence, using $ D_{x},\phi\geq 1$, and $\tilde \Phi=\Vert \Phi^r\Vert_{Q(r',r)}
\leq 
\tilde  D^{6\tilde N^2\tilde\theta}$, we get
\begin{align*}
T_2
=\frac{m(B_{o}(r))}{m(x)}
\leq 2^{18\tilde N}
\tilde  A^9\tilde\Phi^9
\frac{m(B_{x}(r))}{m(x)}
\leq 2^{19\tilde N}\tilde  A^9
\phi^{\frac{\tilde N}{2}}
\tilde D^{\frac{\tilde N}{2}+54\tilde N^2\tilde \theta}.
\end{align*}
Choose 
$$\gamma(r)= 2^{19\tilde N}\tilde  A^9\phi^{\frac{\tilde N}{2}}$$ 
to obtain, with $ \phi\ge 1 $ and $ \tilde A\ge 1 $,
\[\frac{T_1}{\gamma(r)}
\leq 
\frac{ \tilde A^2\tilde D^{6\tilde N^2\tilde \theta}}{\gamma(r)}
\leq \tilde D^{6\tilde N^2\tilde\theta}
\]
and 
\[
\frac{T_2}{\gamma(r)}
\leq 
\frac{2^{19\tilde N}\tilde A^9\phi^{\frac{\tilde N}{2}}\tilde  D^{\frac{\tilde N}{2}+54\tilde N^2\tilde \theta}}{\gamma(r)}
\leq \tilde  D^{\frac{\tilde N}{2}+54\tilde N^2\tilde \theta}.
\]
Hence, we obtain since $\tilde D\geq 1$ and with $ \iota=1/\ln(r/r') \le 1/\ln {4}\le {1}$ as $ r/r' \ge {4}$
\begin{multline*}
\tilde N\vee \ln\left[1\vee
 \frac{T_1}{\gamma(r)}\vee \left(\frac{T_2}{\gamma(r)}\right)^{\iota}
\right]
\leq 
\tilde N\vee \ln \left[\tilde D^{6\tilde N^2\tilde \theta}\vee \left(\tilde  D^{\frac{\tilde N}{2}\iota+54\tilde N^2\tilde\theta\iota}\right)\right]
\\
\leq 
\tilde N\vee \ln \left[ \tilde  D^{\frac{\tilde N}{2}\iota+{54}\tilde N^2\tilde\theta}\right]= 
\tilde N\left[1\vee \ln \tilde  D^{\frac{\iota}{2}+{54}\tilde N\tilde\theta}
\right]=n_o'(r).
\end{multline*}

Finally, we estimate the function $ \phi^{\gamma}(r) =2^{47+\frac{2\tilde N}{\tilde N-2}}\gamma(r)^{\frac{2}{\tilde N}}$ for the $ \gamma(r)= 2^{19\tilde N}\phi^{\frac{\tilde N}{2}}\tilde  A^9$ chosen above with  $ \tilde A=  \Vert A\Vert_{Q_o(r',r)}= 2^{ {43}\tilde N^3}\phi^{8\tilde N^2}$. 
Since $ \tilde N\geq 2$, we have
\begin{align*}
	\phi^{\gamma}(r)& =2^{85+\frac{2\tilde N}{\tilde N-2}+ {774}\tilde N^2}\phi^{144\tilde  N+1}
	\le 2^{ {796}\tilde N^2+\frac{2\tilde N}{\tilde N-2}}\phi^{145\tilde  N}=\phi'(r).
\end{align*}
This then yields $S_{ \phi'}( n', 4R_1,R_2)$ in $o$ since $ \phi'\ge \phi^{\gamma} $ and finishes the proof.
\end{proof}

\begin{proof}[Proof of Theorem~\ref{thm:counting}]
This follows immediately from Theorem~\ref{thm:general} with choosing $ n $ and $ \phi $ to be constant and since  $ \Deg=\deg/m =\deg$	for $ m=1 $.
\end{proof}

\noindent\textbf{Acknowledgments.} The authors are grateful for the financial support of the DFG.

{\scriptsize
\bibliographystyle{alpha}

}
\end{document}